\documentclass[11pt,oneside]{article}
\usepackage[a4paper,centering,bindingoffset=0cm,hmargin=2.5cm,vmargin=2.5cm,includeheadfoot]{geometry}
\usepackage[font=footnotesize,labelfont=bf,width=0.75\textwidth,labelsep=period]{caption}
\usepackage{amsmath}
\usepackage{amssymb,bbm,color}
\usepackage{latexsym}
\usepackage{graphicx}
\usepackage{color,soul}
\usepackage{enumitem}
\usepackage{ctable}

\usepackage[pdftex,colorlinks=true,linkcolor=blue,unicode,pdfpagelabels]{hyperref}

\hypersetup
{
	pdftitle={Mathematical Modelling of Optical Coherence Tomography},
	pdfauthor={Peter Elbau, Leonidas Mindrinos, Otmar Scherzer},
	pdfcreator={Peter Elbau, Leonidas Mindrinos, Otmar Scherzer},
	pdfproducer={Peter Elbau, Leonidas Mindrinos, Otmar Scherzer},
	pdfsubject={Optical Coherence Tomography},
	pdfkeywords={Optical Coherence Tomography} {Scattering Theory}
}
\usepackage[hyperref,amsmath,thmmarks]{ntheorem}
\usepackage{aliascnt}

\let\RE\Re
\let\Re=\undefined
\DeclareMathOperator{\Re}{\RE e}
\let\IM\Im
\let\Im=\undefined
\DeclareMathOperator{\Im}{\IM m}
\newcommand{\R}{\mathbbm R}
\renewcommand{\C}{\mathbbm C}

\newcommand{\Z}{\mathbbm Z}

\newcommand{\Ord}{\mathcal O}
\newcommand{\e}{\mathrm e}
\newcommand{\id}{\mathbbm{1}}
\renewcommand{\i}{\mathrm i}
\renewcommand{\d}{\mathrm d}

\DeclareMathOperator{\laplace}{\Delta_{\it x}}
\DeclareMathOperator{\divx}{div_{\it x}}
\DeclareMathOperator{\divy}{div_{\it y}}
\DeclareMathOperator{\curl}{curl_{\it x}}
\DeclareMathOperator{\grad}{grad_{\it x}}
\DeclareMathOperator{\grady}{grad_{\it y}}
\DeclareMathOperator{\supp}{supp}

\newcommand{\set}[1]{\{ #1 \}}

\newtheorem{lemma}{Lemma}[section]

\newaliascnt{proposition}{lemma}
\newtheorem{proposition}[proposition]{Proposition}
\aliascntresetthe{proposition}

\newaliascnt{assumption}{lemma}
\newtheorem{assumption}[assumption]{Assumption}
\aliascntresetthe{assumption}

\theoremstyle{nonumberplain}
\theoremseparator{:}
\theoremheaderfont{\normalfont\itshape}
\theorembodyfont{\normalfont}
\theoremsymbol{\ensuremath{\square}}
\newtheorem{proof}{Proof}

\makeatletter
\newcommand{\logmessage}[1]{\@latex@warning{#1}}
\makeatother

\begin{document}

\title{Mathematical Modelling of Optical Coherence Tomography}

\author{Peter Elbau$^1$\\{\footnotesize\href{mailto:peter.elbau@univie.ac.at}{peter.elbau@univie.ac.at}}
\and Leonidas Mindrinos$^1$\\{\footnotesize\href{mailto:leonidas.mindrinos@univie.ac.at}{leonidas.mindrinos@univie.ac.at}}     
\and Otmar Scherzer$^{1,2}$\\{\footnotesize\href{mailto:otmar.scherzer@univie.ac.at}{otmar.scherzer@univie.ac.at}}}          

\date{}

\maketitle

\hspace*{1em}
\parbox[t]{0.49\textwidth}{\footnotesize
\hspace*{-1ex}$^1$Computational Science Center\\
University of Vienna\\
Oskar-Morgenstern-Platz 1\\
A-1090 Vienna, Austria}
\parbox[t]{0.4\textwidth}{\footnotesize
\hspace*{-1ex}$^2$Johann Radon Institute for Computational\\
\hspace*{1em}and Applied Mathematics (RICAM)\\
Altenbergerstra{\ss}e 69\\
A-4040 Linz, Austria}

\vspace*{2em}

\abstract{In this chapter a general mathematical model of Optical Coherence Tomography (OCT) is presented on the basis of the electromagnetic theory. 
OCT produces high resolution images of the inner structure of biological tissues. Images are obtained by measuring the time delay and the intensity of the backscattered light from the sample considering also the coherence properties of light. The scattering problem is considered for a weakly scattering medium located far enough from the detector. The inverse problem is to reconstruct the susceptibility of the medium given the measurements for different positions of the mirror. Different approaches are addressed depending on the different assumptions made about the optical properties of the sample. This procedure is applied to a full field OCT system and an extension to standard (time and frequency domain) OCT is briefly presented.

\section{Introduction}
Optical Coherence Tomography (OCT) is a non-invasive imaging technique producing high-resolution images of biological tissues. OCT is based on Low (time) Coherence Interferometry and takes into account the coherence properties of light to image micro-structures with resolution in the range of few micrometers. Standard OCT operates using broadband and continuous wave light in the visible and near-infrared spectrum. OCT images are obtained by measuring the time delay and the intensity of backscattered or back-reflected light from the sample under investigation. 

Since it was first established in 1991 by Huang \textit{et al} \cite{HuaSwaLinSchuStiCha91}, the clinical applications of OCT have been greatly improved. Ophthalmology remains the dominant one, initially applied in 1993 \cite{FerHitDreKamSat93, SwaIzaHeeHuaLinSchu93}. The main reason is that OCT has limited penetration depth in biological tissues, but high resolution. The theory of OCT has been analysed in details in review papers \cite{Fer10, FerDreHitLas03, Pod05, Schm99,  TomWan05}, in book chapters \cite{FerHit02, FerSanJorAnd09, ThoSanMogThrJorJem09} and in books \cite{BouTea02, Bre06, DreFuj08}.

To derive a mathematical model for the OCT system, the scattering properties of the sample need to be described. There exist several different approaches to model the propagation of light within the sample: the radiative transfer equation with scattering and absorption coefficients \cite{Dol98,SchmKnuBon93,TurSerDolKamSha05}, Lambert--Beer's law with the attenuation coefficient \cite{SchmXiaYun98, XuMarDoBop04}, the equations of geometric optics with the refractive index \cite{BruCha05}, and Maxwell's equations with the susceptibility (or the refractive index) as optical parameters of the medium \cite{BroThuBur00,FenWanEld03,KnuSchoBoc96, SchmKnu97, ThrYurAnd00}.
Also statistical approaches using Monte Carlo simulations are used \cite{AndThrYurTycJorFro04,DuaMakYamLimYas11,KirMegKuzSerMyl10,PanBirRosEng95, SmiLinCheNelMil98}.

This chapter describes the propagation of the electromagnetic wave through the sample using Maxwell's equations and adopts the analysis based on the theory of electromagnetic fields scattered by inhomogeneous media \cite{ColKre98, FriWol83}. The sample is hereby considered as a linear dielectric medium (potentially inhomogeneous and anisotropic). Moreover, the medium is considered weakly scattering so that the first order Born approximation can be used and, as it is usually assumed in OCT, the backscattered light is detected far enough from the sample so that the far field approximation is valid. Starting from this model, different reconstruction formulas for special cases regarding the inner structure of the sample are presented. 

This chapter is organised as follows. In \autoref{seBasicPrinciples}, the principles of OCT and different variants of OCT systems are presented. \autoref{seDirectProblem} describes the solution of Maxwell's equations and an appropriate formula for the measurements of OCT is derived. Given the initial field and the optical properties (the susceptibility) of the sample, the solution of the direct problem is obtained in \autoref{seSolutionDirectProblem}. An iterative scheme is derived in the last section for the reconstruction of the unknown susceptibility, which is the inverse problem of OCT.

\section{Basic Principles of OCT}\label{seBasicPrinciples}

OCT is used to gain informations about the light scattering properties of an object by illuminating it with some short laser pulse and measuring the backscattered light. 

The name ``Optical Coherence Tomography'' is motivated by the way the scattering data are measured: To get more precise measurements, the backscattered light is not directly detected, but first superimposed with the original laser pulse and then the intensity of this interference pattern is measured (this means that one measures the ``coherence'' of these two light beams). 

Experimentally, this is done by separating the incoming light at a beam splitter into two identical beams which travel two different paths. One beam is simply reflected by a mirror and sent back to the beam splitter, while the other beam is directed to the sample. At the beam splitter, the beam reflected by the mirror and the backscattered light from the sample are recombined and sent to the detector \cite{FerDreHitLas03, IzaCho08, TomWan05}. See \autoref{picOCTsystem} for an illustration of this procedure.

\begin{figure}[ht]
\begin{center}
\includegraphics{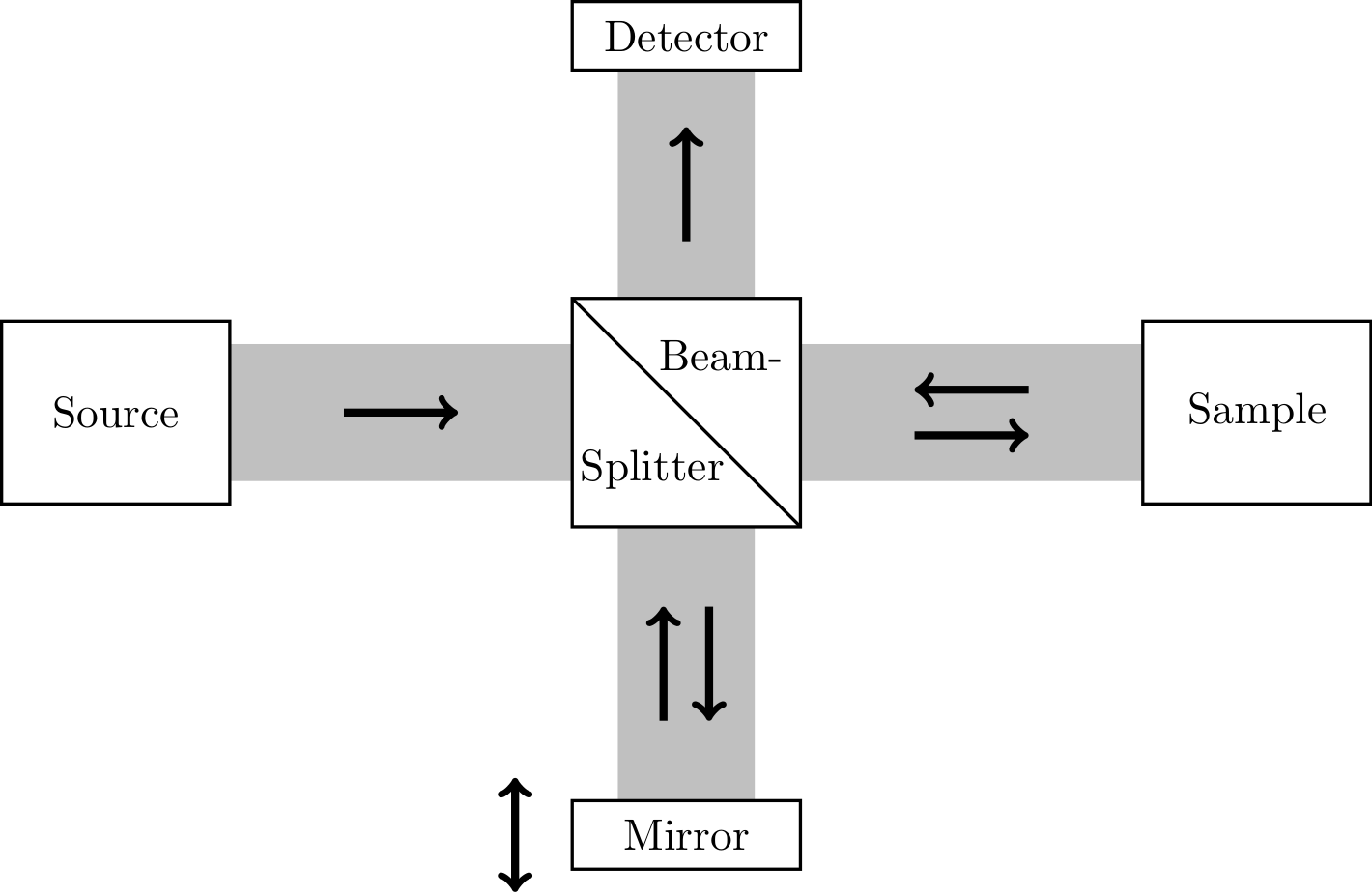}
\caption{Schematic diagram of the light travelling in an OCT system. The laser beam emitted by the source is divided at the beam splitter into two light beams; one is reflected at a mirror, the other one backscattered from the sample. The superposition of the two reflected beams is then measured at the detector.}
%\caption{Standard OCT system, in time domain, based on a Michelson interferometer.}
\label{picOCTsystem}
\end{center}
\end{figure}

There exist different variants of the OCT regarding the way the measurements are done:
\begin{description}
\item[Time and frequency domain OCT:]
In time domain OCT, the position of the mirror is varied and for each position one measurement is performed. On the other hand, in frequency domain OCT, the reference mirror is fixed and the detector is replaced by a spectrometer. Both methods provide equivalent measurements which are connected by a Fourier transform.
\item[Standard and full field OCT:]
In standard OCT, the incoming light is focused through objective lenses to one spot in a certain depth in the sample and the backscattered light is measured in a point detector. This means that to obtain informations of the whole sample, a transversal-lateral scan has to be performed (by moving the light beam over the frontal surface of the sample). 
In full field OCT, the entire frontal surface of the sample is illuminated at once and the single point detector is replaced by a two-dimensional detector array, for instance by a charge-coupled device (CCD) camera.
\item[Polarisation sensitive OCT:]
In classical OCT setups, the electromagnetic wave is simply treated as a scalar quantity. In polarisation sensitive OCT, however, the illuminating light beams are polarised and the detectors measure the intensity of the two polarisation components of the interfered light.
\end{description}
There are also further modifications such as Doppler OCT and quantum OCT, which are not addressed here.
In this chapter, the focus is mainly on time domain full field OCT, but also the others are discussed.

\section{The Direct Scattering Problem}\label{seDirectProblem}
To derive a mathematical model for an OCT system, one has to describe on one hand the propagation and the scattering of the laser beam in the presence of the sample and on the other hand the way how this scattered wave is measured at the detectors. For the first part, the interaction of the incoming light with the sample can be modelled with Maxwell's macroscopic equations.

\subsection{Maxwell's Equations}
Maxwell's equations in matter consist of the partial differential equations
\begin{subequations}\label{eqMaxwell}
\begin{align}
\divx D(t,x) &= 4\pi\rho(t,x),\quad &&t\in\R,\;x\in\R^3, \label{eqMaxwellGauss}\\
\divx B(t,x) &= 0,\quad &&t\in\R,\;x\in\R^3, \\
\curl E(t,x) &= -\frac1c\frac{\partial B}{\partial t}(t,x),\quad &&t\in\R,\;x\in\R^3, \label{eqMaxwellFaraday}\\
\curl H(t,x) &= \frac{4\pi}c J(t,x)+\frac1c\frac{\partial D}{\partial t}(t,x),\quad &&t\in\R,\;x\in\R^3, \label{eqMaxwellAmpere}
\end{align}
\end{subequations}
relating the following physical quantities (at some time $t\in\R$ and some location $x\in\R^3$):
\begin{center}
\begin{tabular}{|c|c|c|}
\hline
speed of light & $c$ & $\R$ \\
\hline
\hline
external charge density			& $\rho(t,x)$	& $\R$ \\
external electric current density	& $J(t,x)$	& $\R^3$\\
\hline
\hline
electric field				& $E(t,x)$	& $\R^3$\\ 
electric displacement			& $D(t,x)$	& $\R^3$ \\ 
magnetic induction			& $B(t,x)$	& $\R^3$\\ 
magnetic field				& $H(t,x)$	& $\R^3$\\ 
\hline
\end{tabular}
\end{center}

Maxwell's equations do not yet completely describe the propagation of the light (even assuming that the charge density $\rho$ and the current density $J$ are known, there are only $8$ equations for the $12$ unknowns $E$, $D$, $B$, and $H$). 

Additionally to Maxwell's equations, it is therefore necessary to specify the relations between the fields $D$ and $E$ as well as between $B$ and $H$.

Let $\Omega\subset\R^3$ denote the domain where the sample is located. It is considered as a non-magnetic, dielectric medium without external charges or currents, this means that for all $t\in\R$ and all $x\in\Omega$ the electric and magnetic fields fulfil the relations
\begin{subequations}\label{eqAssumptionsSample}
\begin{align}
D(t,x)&=E(t,x)+\int_0^\infty\chi(\tau,x)E(t-\tau,x)\d\tau, \label{eqSusceptibility}\\
B(t,x)&=H(t,x), \\
\rho(t,x)&=0, \\
J(t,x)&=0,
\end{align}
\end{subequations}
where the function $\chi:\R\times\R^3\to\R^{3\times 3}$ (for convenience, $\chi$ is also defined for negative times by $\chi(t,x)=0$ for $t<0$, $x\in\R^3$) is called the (electric) susceptibility and is the quantity to be recovered.  The time dependence of $\chi$ hereby describes the fact that a change in the electric field $E$ cannot immediately cause a change in the electric displacement $D$. Since this delay is quite small, it is sometimes ignored and $\chi(t,x)$ is then replaced by $\delta(t)\chi(x)$. Moreover, the medium is often considered to be isotropic, which means that $\chi$ is a multiple of the identity matrix.

The sample is situated in vacuum and the assumptions~\eqref{eqAssumptionsSample} are modified by setting for all $t\in\R$ and all $x\in\R^3\setminus\Omega$
\begin{subequations}\label{eqAssumptionsVacuum}
\begin{align}
D(t,x)&=E(t,x),\label{eqSusceptibilityVacuum} \\
B(t,x)&=H(t,x), \\
\rho(t,x)&=0, \\
J(t,x)&=0.
\end{align}
\end{subequations}
This simply corresponds to extend the equations \eqref{eqAssumptionsSample} to $\R\times\R^3$ and to assume $\chi(t,x)=0$ for all $t\in\R$, $x\in\R^3\setminus\Omega$. 

In this case of a non-magnetic medium, Maxwell's equations result into one equation for the electric field $E$. To get rid of the convolution in \eqref{eqSusceptibility}, it is practical to consider the Fourier transform with respect to time. In the following, the convention
\[ \hat f(\omega,x) = \int_{-\infty}^\infty f(t,x)\e^{\i\omega t}\d t ,\]
for the Fourier transform of a function $f$ with respect to $t$ is used.

\begin{proposition}\label{thVectorHelmholtz}
Let $E$, $D$, $B$, and $H$ fulfil Maxwell's equations~\eqref{eqMaxwell}. Moreover, let assumptions~\eqref{eqAssumptionsSample} and~\eqref{eqAssumptionsVacuum} be satisfied. 
Then the Fourier transform $\hat E$ of $E$ fulfils the vector Helmholtz equation
\begin{equation}\label{eqVectorHelmholtz}
\curl \curl\hat E (\omega,x) - \frac{\omega^2}{c^2}(\id+\hat\chi(\omega,x))\hat E(\omega,x) = 0,\quad\omega\in\R,\;x\in\R^3.
\end{equation}
\end{proposition}
\begin{proof}
Applying the curl to~\eqref{eqMaxwellFaraday} and using \eqref{eqMaxwellAmpere} with the assumptions $B=H$ and $J=0,$ yields
\begin{equation}\label{eqVectorHelmholtzFourier}
\curl \curl E  (t,x)= -\frac1c\frac{\partial\curl B}{\partial t} (t,x)= -\frac1{c^2}\frac{\partial^2D}{\partial t^2}(t,x).
\end{equation}
The Fourier transform of \eqref{eqSusceptibility} and~\eqref{eqSusceptibilityVacuum} and the Fourier convolution theorem (recall that $\chi$ is set to zero outside $\Omega$) imply that
\[ \hat D(\omega,x) = (\id+\hat\chi(\omega,x))\hat E(\omega,x),\quad\text{for all}\quad \omega \in\R,\;x\in\R^3. \]
Therefore, the equation \eqref{eqVectorHelmholtz} follows by taking the Fourier transform of \eqref{eqVectorHelmholtzFourier}.
\end{proof}

\subsection{Initial Conditions}

The sample is illuminated with a laser beam described initially (before it interacts with the sample) by the electric field $E^{(0)}:\R\times\R^3\to\R^3$ which is (together with some magnetic field) a solution of Maxwell's equations \eqref{eqMaxwell} with the assumptions \eqref{eqAssumptionsVacuum} for all $x \in \R^3$. Then, it follows from the proof of the \autoref{thVectorHelmholtz}, for $\chi = 0 $, that
\begin{equation}\label{eqVectorHelmholtzIncident}
\curl \curl\hat E^{(0)} (\omega,x) - \frac{\omega^2}{c^2} \hat E^{(0)} (\omega,x) = 0,\quad\omega\in\R,\;x\in\R^3.
\end{equation}

Moreover, it is assumed that $E^{(0)}$ does not interact with the sample until the time $t=0$, which means that $\supp E^{(0)}(t,\cdot)\cap\Omega = \emptyset$ for all $t\le0$. 

The electric field $E:\R\times\R^3\to\R^3$ generated by this incoming light beam in the presence of the sample is then a solution of Maxwell's equations \eqref{eqMaxwell} with the assumptions \eqref{eqAssumptionsSample} and the initial condition
\begin{equation}\label{eqInitialCondition}
E(t,x) = E^{(0)}(t,x)\quad\text{for all}\quad t\le0,\;x\in\R^3.
\end{equation}

Since Maxwell's equations for $E$ in \autoref{thVectorHelmholtz} are reformulated as an equation for the Fourier transform $\hat E$, it is helpful to rewrite the initial condition in terms of $\hat E$. 

\begin{proposition}\label{thPaleyWiener}
Let $E$ (together with some magnetic field $H$) fulfil Maxwell's equations~\eqref{eqMaxwell} with the assumptions~\eqref{eqAssumptionsSample} and~\eqref{eqAssumptionsVacuum} and with 
the initial condition~\eqref{eqInitialCondition}.

Then the Fourier transform of $E-E^{(0)}$ fulfils that the function $\omega\mapsto \hat E(\omega,x)-\hat E^{(0)}(\omega,x)$, defined on $\R$, can be extended to a square integrable, holomorphic function on the upper 
half plane $\mathbbm H=\{\omega\in\C\mid\Im(\omega)>0\}$ for every $x\in\R^3$.
\end{proposition}
\begin{proof}
From the initial condition~\eqref{eqInitialCondition} it follows that $E(t,x)-E^{(0)}(t,x)=0$ for all $t\le0 .$ Thus, the result is a direct consequence form the Paley--Wiener theorem, which is based on the fact that in this case
\[ \hat E(\omega,x)-\hat E^{(0)}(\omega,x) = \int_0^\infty(E-E^{(0)})(t,x)\e^{\i\omega t}\d t \]
is well defined for all $\omega\in\mathbbm H$ and complex differentiable with respect to $\omega\in\mathbbm H$.
\end{proof}

Remark that the electric field $E$ is uniquely defined by \eqref{eqVectorHelmholtz} and~\autoref{thPaleyWiener}.

\subsection{The Measurements}\label{seMeasurementsGeneral}

The measurements are obtained by the combination of the backscattered field from the sample and the back-reflected field from the mirror. In practice, see \autoref{picOCTsystem}, the sample and the mirror are in different positions. However, without loss of generality, a placement of them around the origin is assumed in the proposed formulation, in order to avoid rotating the coordinate system. To do so, the simultaneously illumination of the sample and the mirror is suppressed and two different illumination schemes are considered. The gain is to keep the same coordinate system but the reader should not be confused with illumination at different times.

Thus, the electric field $E$, which is obtained by illuminating the sample with the initial field $E^{(0)}$ (that is $E$ solves \eqref{eqVectorHelmholtz} 
with the initial condition~\eqref{eqInitialCondition}), is combined with $E_r$ which is the electric field obtained by replacing the sample by a mirror and illuminating with the same initial field $E^{(0)}$.

The mirror is placed orthogonal to the unit vector $e_3=(0,0,1)$ through the point $re_3$. As in \eqref{eqInitialCondition}, it is assumed that $\supp E^{(0)}(t,\cdot)$ 
does not interact with the mirror for $t<0$, so that
\begin{equation}\label{eqInitialConditionMirror}
E_r(t,x)=E^{(0)}(t,x)\quad\text{for all}\quad t<0,\;x\in\R^3.
\end{equation}
Then the resulting electric field $E_r:\R\times\R^3\to\R^3$ is given as the solution of the same equations as $E$ (Maxwell's equations \eqref{eqMaxwell} together with the assumptions \eqref{eqAssumptionsSample} and initial condition \eqref{eqInitialConditionMirror}) with the susceptibility $\chi$ replaced by the susceptibility $\chi_r$ of the mirror at position $r$. 
One sort of (ideal) mirror can be described via the susceptibility $\chi_r(t,x)=0$ for $x_3>r$ and $\chi_r(t,x)=C\delta(t)\mathbbm1$ for $x_3\le r$ with an (infinitely) large constant $C>0$.

\begin{figure}[ht]
\begin{center}
\includegraphics{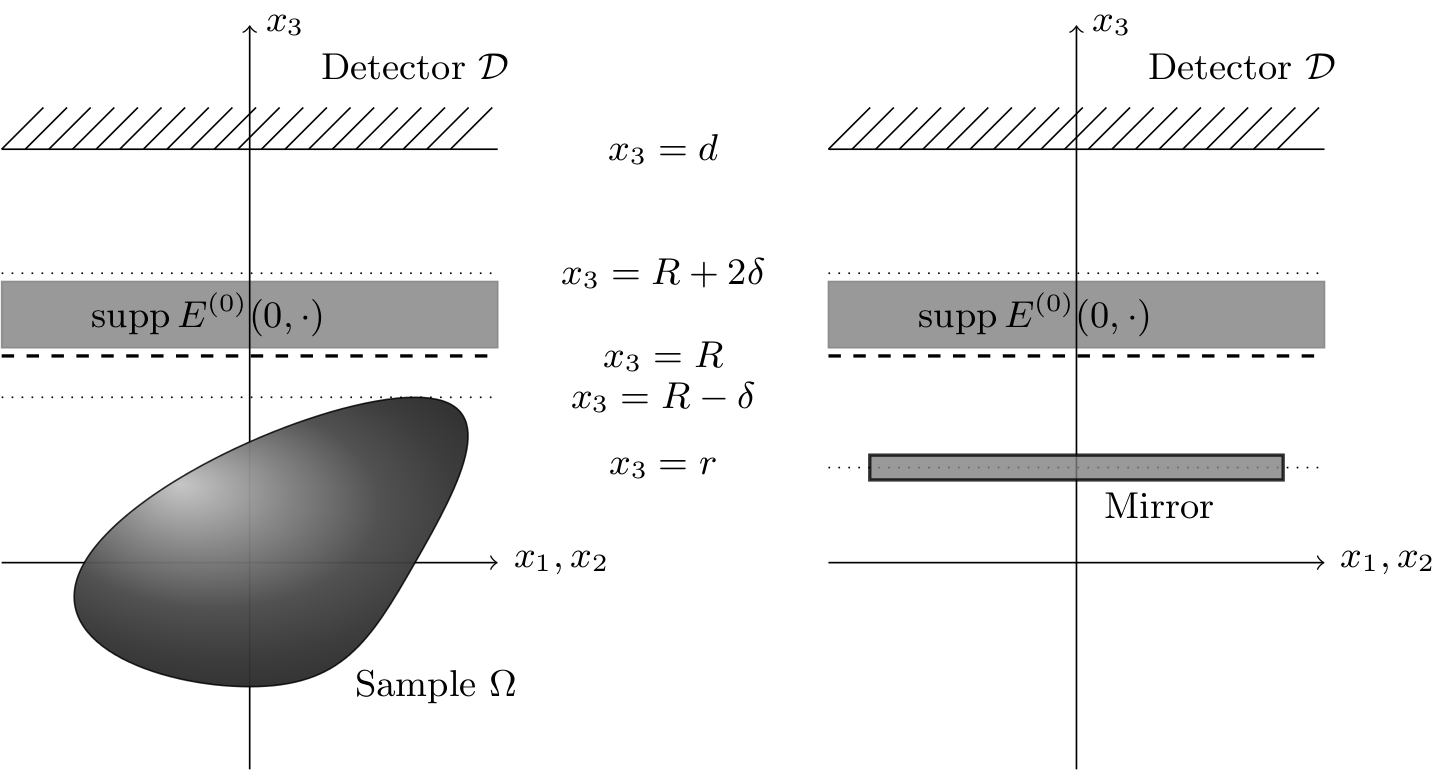}
\caption{The two scattering problems involved in OCT: On the left hand side the scattering of the initial wave on the sample $\Omega$; on the right hand side the reference problem where the initial wave $E^{(0)}$ is reflected by a perfect mirror at a tunable position $r\in(-\infty,R)$. The two resulting electric fields, $E$ and $E_r$, are then combined and this superposition $E+E_r$ is measured at the detector surface $\mathcal D$.}
\label{fgMeasurementSetup}
\end{center}
\end{figure}

The intensity $I_r$ of each component of the superposition of the electric fields $E$ and $E_r$ averaged over all time is measured at some detector points. The detectors are positioned at all points on the plane 
\[ \mathcal D=\{x\in\R^3\mid x_3=d\} \]
parallel to the mirror at a distance $d>0$ from the origin. The mirror and the sample are both located in the lower half plane of the detector surface with some minimal distance to~$\mathcal D$. Moreover, the highest possible position $R\in(\delta,d-2\delta)$ of the mirror shall be by some distance $\delta>0$ closer to the detector than the sample, this means (see \autoref{fgMeasurementSetup})
\begin{equation}\label{eqGeometryAssumptions}
\sup_{x\in\Omega}x_3<R-\delta\quad\text{and}\quad r\in(-\infty,R).
\end{equation}

To simplify the argument, let us additionally assume that the incoming electric field $E^{(0)}$ does not influence the detector after the time $t=0$, meaning that
\begin{equation}\label{eqNoSignalFromIncomingWave}
E^{(0)}(t,x)=0\quad\text{for all}\quad t\ge0,\;x\in\mathcal D.
\end{equation}

At the detector array, the data are obtained by measuring
\begin{equation}\label{eqIntensity}
I_{r,j}(x)=\int_0^\infty|E_j(t,x)+E_{r,j}(t,x)|^2\d t,\quad x\in\mathcal D,\;j\in\{1,2,3\}.
\end{equation}
In standard OCT, the polarisation is usually ignored. In this case, only the total intensity $I_r=\sum_{j=1}^3I_{r,j}$ needs to be measured, see \autoref{seIsotropicCase} for the corresponding reconstruction formulas.

In this measurement setup, it is easy to acquire besides the intensity $I_r$ also the intensity of the two waves $E$ and $E_r$ separately by blocking one of the two waves $E$ and $E_r$ at a time. Practically, it is sometimes not even necessary to measure them since the intensity of the reflected laser beam $E_r$ can be explicitly calculated from the knowledge of the initial beam $E^{(0)}$, and the intensity of $E$ is usually negligible compared with the intensity $I_r$ (because of the assumption \eqref{eqNoSignalFromIncomingWave}, the field $E$ contains only backscattered light at the detector after the measurement starts). Therefore, one can consider instead of $I_r$ the function
\begin{equation}\label{eqEffectiveMeasurements}
M_{r,j}(x) = \frac12\left(I_{r,j}-\int_0^\infty|E_j(t,x)|^2\d t -\int_0^\infty|E_{r,j}(t,x)|^2\d t\right)
\end{equation}
for $r\in(-\infty,R)$, $j\in\{1,2,3\}$, and $x\in\mathcal D$ as the measurement data.

\begin{proposition}\label{thMeasurements}
Let the initial conditions~\eqref{eqInitialCondition} and~\eqref{eqInitialConditionMirror} and the additional assumption~\eqref{eqNoSignalFromIncomingWave} be satisfied. Then, for all $x\in\mathcal D$, $r\in(-\infty,R)$, and $j\in\{1,2,3\}$ the measurements $M_r$, defined by \eqref{eqEffectiveMeasurements}, fulfil
\begin{subequations}
\begin{align}
M_{r,j}(x) &= \int_{-\infty}^\infty(E_j-E_j^{(0)})(t,x)(E_{r,j}-E_j^{(0)})(t,x)\d t \label{eqEffectiveMeasurementsForwardModel} \\
&= \int_{-\infty}^\infty(\hat E_j-\hat E_j^{(0)})(\omega,x)\overline{(\hat E_{r,j}-\hat E_j^{(0)})(\omega,x)}\d\omega, \label{eqEffectiveMeasurementsForwardModel2}
\end{align}
\end{subequations}
\end{proposition}
\begin{proof}
Expanding the function $I_{r,j}$, given by \eqref{eqIntensity}, gives
\[ I_{r,j}(x) = \int_0^\infty\big(|E_j(t,x)|^2+|E_{r,j}(t,x)|^2+2E_j(t,x)E_{r,j}(t,x)\big)\d t. \]
Thus, by the definition \eqref{eqEffectiveMeasurements} of $M_r$, it follows that
\[ M_{r,j}(x) = \int_0^\infty E_j(t,x)E_{r,j}(t,x)\d t, \]
which, using the assumption \eqref{eqNoSignalFromIncomingWave}, can be rewritten in the form 
\[ M_{r,j}(x) = \int_0^\infty(E_j-E_j^{(0)})(t,x)(E_{r,j}-E_j^{(0)})(t,x)\d t. \]
Then, since $E$ and $E_r$ coincide with $E^{(0)}$ for $t<0$, see~\eqref{eqInitialCondition} and~\eqref{eqInitialConditionMirror}, the integration can be extended to all times. This proves the formula \eqref{eqEffectiveMeasurementsForwardModel} for $M_r$. The second formula follows from Plancherel's theorem.
\end{proof}

\section{Solution of the Direct Problem}\label{seSolutionDirectProblem}

In this section the solution of the direct problem, to determine the measurements $M_r$, defined by \eqref{eqEffectiveMeasurementsForwardModel}, from the susceptibility $\chi$,  is derived using Born and far field approximation for the electric field. 

\begin{proposition}
Let $E$ be a solution of the equations~\eqref{eqVectorHelmholtz} and~\eqref{eqInitialCondition}. Then, the Fourier transform $\hat E$ solves the Lippmann--Schwinger integral equation
\begin{equation}\label{eqLippmannSchwinger}
\hat E(\omega,x) = \hat E^{(0)}(\omega,x)+\left(\frac{\omega^2}{c^2} \id +\grad\divx\right)\int_{\R^3}G (\omega, x-y)\hat\chi(\omega,y)\hat E(\omega,y)\d y ,
\end{equation}
where $G$ is the fundamental solution of the Helmholtz equation given by
\[
G (\omega, x) = \frac{\e^{\i\frac\omega c|x|}}{4\pi|x|}, \quad x \neq 0, \, \omega \in \R.
\]
\end{proposition}
\begin{proof}
Equation \eqref{eqVectorHelmholtz} can be rewritten in the form
\[ \curl \curl\hat E (\omega,x)-\frac{\omega^2}{c^2}\hat E(\omega,x) = \phi (\omega,x) \]
with the inhomogeneity
\begin{equation}\label{eqInhomogeneity}
\phi (\omega,x) = \frac{\omega^2}{c^2}\hat\chi(\omega,x)\hat E(\omega,x).
\end{equation}

Using that $\hat E^{(0)}$ solves \eqref{eqVectorHelmholtzIncident}, the difference $\hat E-\hat E^{(0)}$  satisfies the inhomogeneous vector Helmholtz equation
\begin{equation}\label{eqVectorWaveEquation}
\curl \curl(\hat E-\hat E^{(0)}) (\omega,x)-\frac{\omega^2}{c^2}(\hat E-\hat E^{(0)})(\omega,x) = \phi (\omega,x).
\end{equation}
The divergence of this equation, using that $\divx \curl (\hat E-\hat E^{(0)}) =0 ,$ implies 
\begin{equation}\label{eqDivergenceOfWaveEquation}
\divx(\hat E-\hat E^{(0)})(\omega,x) = -\frac{c^2}{\omega^2}\divx \phi (\omega,x).
\end{equation}
Applying the vector identity 
\[ \curl \curl(\hat E-\hat E^{(0)}) = \grad\divx(\hat E-\hat E^{(0)})-\laplace(\hat E-\hat E^{(0)}) \]
in \eqref{eqVectorWaveEquation} and using \eqref{eqDivergenceOfWaveEquation} yields 
\[ \laplace(\hat E-\hat E^{(0)})(\omega,x)+\frac{\omega^2}{c^2}(\hat E-\hat E^{(0)})(\omega,x) = -\frac{c^2}{\omega^2}\grad\divx \phi (\omega,x)-\phi (\omega,x). \]
This is a Helmholtz equation for $\hat E-\hat E^{(0)}$ and the general solution which is (with respect to~$\omega$) holomorphic in the upper half plane (equivalent to \eqref{eqInitialCondition} by \autoref{thPaleyWiener}) is given by, see~\cite{ColKre98}
\begin{align*}
(\hat E-\hat E^{(0)})(\omega,x) &= -\frac{c^2}{\omega^2}\int_{\R^3}G (\omega, x-y)\left(\frac{\omega^2}{c^2}\id+\grady\divy\right)\phi (\omega,y)\d y \\
&= -\frac{c^2}{\omega^2}\left(\frac{\omega^2}{c^2}\id+\grad\divx\right)\int_{\R^3}G (\omega ,x-y)\phi (\omega,y)\d y.
\end{align*}
For the last equality, integration by parts and $\grad G(\omega, x-y) =  - \grady G(\omega, x-y)$ were used. The Lippmann--Schwinger equation~\eqref{eqLippmannSchwinger} follows from the last expression inserting the expression~\eqref{eqInhomogeneity} for $\phi.$
\end{proof}

This integral equation uniquely defines the electric field $E.$ The reader is refered to~\cite{AmmBao01, ColKre98} for the isotropic case and to \cite{Pot00} for an anisotropic medium. 

\subsection{Born and Far Field Approximation}
To solve the Lippmann--Schwinger equation~\eqref{eqLippmannSchwinger}, the medium is assumed to be weakly scattering, which means that $\hat\chi$ is sufficiently small (implying that the difference $E-\hat E^{(0)}$ becomes small compared to $E^{(0)}$) so that the Born approximation $E^{(1)}$, defined by
\begin{equation}\label{eqBornApproximation}
\hat E^{(1)}(\omega,x) = \hat E^{(0)}(\omega,x)+\left(\frac{\omega^2}{c^2} \id+\grad\divx\right)\int_{\R^3}G (\omega, x-y)\hat\chi(\omega,y)\hat E^{(0)}(\omega,y)\d y,
\end{equation}
is considered a good approximation for the electric field $E$, see \cite{BorWol99}. To describe multiple scattering events, one considers higher order Born approximations. For different linearisation techniques the reader is referred to \cite{AmmBao01, Hoh06}.
Moreover, since the detector in OCT is typically quite far away from the sample, one can simplify the expression (\ref{eqBornApproximation}) for the electric field at the detector array by replacing it with its asymptotic behaviour for $|x|\to\infty$, that is replace the formula for $E^{(1)}$ by its far field approximation (the far field approximation could also be applied to the solution $E$ of the Lippmann--Schwinger equation \eqref{eqLippmannSchwinger}). 
\begin{proposition}\label{thAsymptoticExpansionGreenFunction}
Consider, for a given function $\phi :\R^3\to\R^3$ with compact support and some parameter $a\in\R ,$ the function
\[ g:\R^3\to\R^3,\quad g(x) = \int_{\R^3}\frac{\e^{\i a|x-y|}}{|x-y|}\phi (y)\d y. \]
Then, it follows, asymptotically for $\rho\to\infty$ and uniformly in $\vartheta\in S^2,$ that 
\begin{equation}\label{eqAsymptoticExpansionGreenFunction}
(a^2+\grad\divx) g(\rho\vartheta) \simeq-\frac{a^2\e^{\i a\rho}}\rho\int_{\R^3}\vartheta\times(\vartheta\times \phi (y))\e^{-\i a\left<\vartheta,y\right>}\d y
\end{equation}
\end{proposition}
\begin{proof}
Consider the function
\[ \Gamma :\R^3\to\C,\quad \Gamma (x) = \frac{\e^{\i a|x|}}{|x|}. \]
Then
\begin{align*}
\frac{\partial^2 \Gamma}{\partial x_j\partial x_k}(x) &= \frac{\partial}{\partial x_j}\left[\left(\frac{\i a}{|x|^2}-\frac1{|x|^3}\right)x_k\e^{\i a|x|}\right] \\
&= \left[\left(\frac{\i a}{|x|^2}-\frac1{|x|^3}\right)\delta_{jk}+\left(\frac{\i a}{|x|^2}-\frac1{|x|^3}\right)\frac{\i ax_jx_k}{|x|}+\left(-2\frac{\i a}{|x|^3}+3\frac1{|x|^4}\right)\frac{x_jx_k}{|x|}\right]\e^{\i a|x|}.
\end{align*}
Therefore, writing $x$ in spherical coordinates: $x=\rho\vartheta$ with $\rho>0$, $\vartheta\in S^2$, for $\rho\to\infty$ uniformly in $\vartheta,$ it can be seen that
\[ \frac{\partial^2 \Gamma}{\partial x_j\partial x_k}(\rho\vartheta) = -\frac{a^2\e^{\i a\rho}}\rho\vartheta_j\vartheta_k+\Ord \left( \frac1{\rho^2}\right) , \]
The approximation (local uniformly in $y\in\R^3)$
\[ |\rho\vartheta-y| = \rho\sqrt{|\vartheta|^2-\frac2\rho\left<\vartheta,y\right>+\frac1{\rho^2}|y|^2} = \rho-\left<\vartheta,y\right>+\Ord \left( \frac1\rho \right) , \]
implies that (again uniformly in $\vartheta\in S^2$)
\[ \frac{\partial^2 \Gamma}{\partial x_j\partial x_k}(\rho\vartheta-y) = -\frac{a^2\e^{\i a(\rho-\left<\vartheta,y\right>)}}\rho\vartheta_j\vartheta_k+\Ord \left( \frac1\rho \right) . \]
Now, considering the compact support of $\phi$ and using that $x\in\R^3\setminus\supp \phi$
\[ (\grad\divx g)_j(x) = \sum_{k=1}^3\frac\partial{\partial x_j}\int_{\R^3}\frac{\partial \Gamma}{\partial x_k}(x-y)\phi_k(y)\d y = \int_{\R^3} \sum_{k=1}^3\frac{\partial^2 \Gamma}{\partial x_j\partial x_k}(x-y)\phi_k(y)\d y. \]
Asymptotically for $|x|\to\infty$ (again using the compact support of $\phi$) one obtains
\[ a^2g_j(\rho\vartheta)+(\grad\divx g)_j(\rho\vartheta) \simeq a^2\int_{\R^3} \sum_{k=1}^3 \frac{\e^{\i a(\rho-\left<\vartheta,y\right>)}}\rho\left(\delta_{jk}-\vartheta_j\vartheta_k\right)\phi_k(y)\d y. \]
The approximation \eqref{eqAsymptoticExpansionGreenFunction} follows from the vector identity $\vartheta\times (\vartheta\times \phi) = \left<\vartheta,\phi \right>\vartheta-|\vartheta|^2 \phi$ and $|\vartheta|=1.$
\end{proof}

The application of both the far field and the Born approximation, this means \autoref{thAsymptoticExpansionGreenFunction} for the expression~\eqref{eqBornApproximation} of $E^{(1)},$ that is setting $a=\omega /c$ and $\phi = \tfrac1{4\pi}\hat \chi \hat E^{(0)}$ in \autoref{thAsymptoticExpansionGreenFunction}, imply the asymptotic behaviour
\begin{equation}\label{eqFarField}
\hat E^{(1)}(\omega,\rho\vartheta) \simeq\hat E^{(0)}(\omega,\rho\vartheta)-\frac{\omega^2\e^{\i\frac\omega c\rho}}{4\pi \rho c^2}\int_{\R^3}\vartheta\times\big(\vartheta\times(\hat\chi(\omega,y)\hat E^{(0)}(\omega,y))\big)\e^{-\i\frac\omega c\left<\vartheta,y\right>}\d y.
\end{equation}

\subsection{The Forward Operator}
To obtain a forward model for the measurements described in \autoref{seMeasurementsGeneral}, the (approximative) formula \eqref{eqFarField} is considered as the model for the solution of the scattering problem. To make this formula concrete, one has to plug in a function $E^{(0)}$ describing the initial illumination (recall that $E^{(0)}$ has to solve~\eqref{eqVectorHelmholtzIncident}).

The specific illumination is a laser pulse propagating in the direction $-e_3$, orthogonal to the detector surface $\mathcal D=\{x\in\R^3\mid x_3=d\}$, this means
\begin{equation}\label{eqInitialIllumination}
E^{(0)}(t,x) = f(t+\tfrac{x_3}c)p,
\end{equation} 
which solves Maxwell's equations \eqref{eqMaxwell} with the assumptions \eqref{eqAssumptionsVacuum} for some fixed vector $p\in \R^3, $ with $p_3=\left<p,e_3\right>=0,$ describing the polarisation of the initial laser beam. 

\begin{proposition}
The function $E^{(0)}$, defined by \eqref{eqInitialIllumination} with $\left<p,e_3\right>=0$, solves together with the magnetic field $H^{(0)}$, defined by
\[ H^{(0)}(t,x) = f(t+\tfrac{x_3}c)p\times e_3, \]
Maxwell's equations \eqref{eqMaxwell} in the vacuum, that is with the additional assumptions \eqref{eqAssumptionsVacuum}.
\end{proposition}
\begin{proof}
The four equations of \eqref{eqMaxwell} can be directly verified:
\begin{align*}
\divx E^{(0)}(t,x) &= \frac1cf'(t+\tfrac{x_3}c)\left<e_3,p\right> = 0, \\
\divx H^{(0)}(t,x) &= \frac1cf'(t+\tfrac{x_3}c)\left<e_3,p\times e_3\right> = 0, \\
\curl E^{(0)}(t,x) &= \frac1cf'(t+\tfrac{x_3}c)e_3\times p = -\frac1c\frac{\partial H^{(0)}}{\partial t}(t,x), \\
\curl H^{(0)}(t,x) &= \frac1cf'(t+\tfrac{x_3}c)e_3\times(p\times e_3) = \frac1cf'(t+\tfrac{x_3}c)p = \frac1c\frac{\partial E^{(0)}}{\partial t}(t,x).
\end{align*}
\end{proof}

To guarantee that the initial field $E^{(0)}$ (and also the magnetic field $H^{(0)}$) does not interact with the sample or the mirror for $t\le0$ and neither contributes to the measurement at the detectors for $t\ge0$ as required by \eqref{eqInitialConditionMirror} and \eqref{eqNoSignalFromIncomingWave} the vertical distribution $f:\R\to\R$ should satisfy (see \autoref{fgMeasurementSetup})
\begin{equation}\label{eqInitialIlluminationPulseLength1} 
\supp f\subset(\tfrac Rc,\tfrac dc). 
\end{equation}

In the case of an illumination $E^{(0)}$ of the form \eqref{eqInitialIllumination}, the electric field $E_r$ produced by an ideal mirror at the position $r$ is given by
\begin{equation}\label{eqReflectedField}
E_r(t,x) = \begin{cases}\big(f(t+\frac{x_3}c)-f(t+ \frac{x_3}c + 2\,\frac{r-x_3}c) \big)p&\text{if}\;x_3>r,\\0&\text{if}\;x_3\le r.\end{cases}
\end{equation}
This just corresponds to the superposition of the initial wave with the (orthogonally) reflected wave, which travels additionally the distance $2\tfrac{x_3 - r}c$. The change in polarisation of the reflected wave (from $p$ to $-p$) 
comes from the fact that the tangential components of the electric field have to be continuous across the border of the mirror. 

The following proposition gives the form of the measurements $M_r$, described in \autoref{seMeasurementsGeneral}, on the detector surface $\mathcal D$ for the specific illumination \eqref{eqInitialIllumination}.

\begin{proposition}\label{thElectricFieldFromMeasurements1}
Let $E^{(0)}$ be an initial illumination of the form~\eqref{eqInitialIllumination}  satisfying \eqref{eqInitialIlluminationPulseLength1}. Then, the equations for the measurements $M_r$ from \autoref{thMeasurements} are given by
\begin{subequations}
\begin{align}
M_{r,j}(x) &=  -p_j\int_{-\infty}^\infty(E_j-E_j^{(0)})(t,x)f(t+\tfrac{2r-x_3}c)\d t, \label{eqMeasurementsPlaneWave1} \\
&= -\frac{p_j}{2\pi}\int_{-\infty}^\infty(\hat E_j-\hat E_j^{(0)})(\omega,x)\hat f(-\omega)\e^{\i\frac\omega c(2r-x_3)}\d\omega \label{eqElectricFieldFromMeasurements1} 
\end{align}
\end{subequations}
for all $j\in\{1,2,3\}$, $r\in (-\infty,R)$, and $x\in\mathcal D$.
\end{proposition}

\begin{proof}
Since the electric field $E_r$ reflected on a mirror at vertical position $r\in(-\infty,R)$ is according to \eqref{eqReflectedField} given by
\[ E_r(t,x) = \big(f(t+\tfrac{x_3}c)-f(t+\tfrac{2r-x_3}c)\big)p\quad\text{for all}\quad t\in\R,\;x\in\mathcal D, \]
the measurement functions $M_r$ (defined by \eqref{eqEffectiveMeasurements} and computed with  \eqref{eqEffectiveMeasurementsForwardModel}) are simplified, for the particular initial illumination $E^{(0)}$ of the form 
\eqref{eqInitialIllumination}, to \eqref{eqMeasurementsPlaneWave1} for $x \in \mathcal D.$

Since formula \eqref{eqMeasurementsPlaneWave1} is just a convolution, the electric field $E-E^{(0)}$ can be rewritten, in terms of its Fourier transform, in the form
\[ M_{r,j}(x) = -\frac{p_j}{2\pi}\int_{-\infty}^\infty\int_{-\infty}^\infty(\hat E_j-\hat E_j^{(0)})(\omega,x)\e^{-\i\omega t}f(t+\tfrac{2r-x_3}c)\d\omega\d t. \]
Interchanging the order of integration and applying the Fourier transform $\hat f$ of $f,$ it follows equation \eqref{eqElectricFieldFromMeasurements1}.
\end{proof}

In the limiting case of a delta impulse as initial wave, that is for $f(\xi)=\delta(\xi-\xi_0)$ with some constant $\xi_0\in(\frac Rc,\frac dc)$ satisfying \eqref{eqInitialIlluminationPulseLength1}, the measurements provide directly the electric field. 
Indeed, it can be seen from \eqref{eqMeasurementsPlaneWave1} that
\[ M_{r,j}(x) = -p_j(E_j-E_j^{(0)})(\tfrac{x_3-2r}c+\xi_0,x). \] 
By varying $r\in(-\infty,R)$, the electric field $E$ can be obtained (to be more precise, its component in direction of the initial polarisation) as a function of time at every detector position.

The following assumptions are made:
\begin{assumption}\label{enAssumption1}
The susceptibility $\chi$ is sufficiently small so that the Born approximation~$E^{(1)}$ for the solution $E$ of the Lippmann--Schwinger equation~\eqref{eqLippmannSchwinger} can be applied.
\end{assumption}
\begin{assumption}\label{enAssumption2}
The detectors are sufficiently far away from the object so that one can use the far field asymptotics~\eqref{eqFarField} for the measured field.
\end{assumption}

Under these assumptions, one can approximate the electric field by the far field expression of the Born approximation $E^{(1)}$ and plug in the expression in \eqref{eqFarField} to obtain the measurements $M_{r,j}, \, j \in \set{1,2,3}.$

The above analysis, introducing appropriate operators, can then be formulated as an operator equation. The integral equation \eqref{eqFarField} can be formally written as
\[
(\hat E^{(1)} - \hat E^{(0)})(\omega,x) = (\mathcal{K}_0 \hat\chi) (\omega, x),
\]
for a given $\hat E^{(0)}$ where the operator $\mathcal{K}_0  : \hat\chi \mapsto \hat E^{(1)} - \hat E^{(0)}$ is given by
\[
(\mathcal{K}_0 v) (\omega, \rho \vartheta) = -\frac{\omega^2\e^{\i\frac\omega c\rho}}{4\pi \rho c^2}\int_{\R^3}\vartheta\times\big(\vartheta\times(v(\omega,y)\hat E^{(0)}(\omega,y))\big)\e^{-\i\frac\omega c\left<\vartheta,y\right>}\d y, \quad \rho >0, \, \vartheta \in S^2.
\]
It is emphasized that $\hat\chi : \R \times \R^3 \rightarrow \C^{3 \times 3}$ and $E^{(1)} - \hat E^{(0)} : \R \times \R^3 \rightarrow \C^3,$ that is $\mathcal{K}_0 v$ is a function from $\R \times \R^3$ into $\C^3.$ Equivalently, considering the equation \eqref{eqEffectiveMeasurementsForwardModel2}, one has
\[
M (r,x) = \left( M_{r,j} (x) \right)_{j=1}^3 = \big(\mathcal{M}  (\hat E^{(1)} - \hat E^{(0)})\big) (r, x),
\]
where the operator $\mathcal{M}$ is defined by
\[
(\mathcal{M} v) (r, x) = \left( \int_{-\infty}^\infty v_j (\omega,x)\overline{(\hat E_{r,j}-\hat E_j^{(0)})(\omega,x)}\d\omega\right)_{j=1}^3 , \quad x \in \mathcal{D}.
\]
Here, $\mathcal Mv$ is a function from $\R \times \mathcal D$ to $\R^3.$ Thus, combining the operators $\mathcal{K}_0$ and $\mathcal M$, the forward operator $\mathcal{F} : \hat\chi \mapsto M,$ $\mathcal{F}= \mathcal{M} \mathcal{K}_0 $ models the direct problem. The inverse problem of OCT is then formulated as an operator equation
\begin{equation}\label{eqInverseProblem}
\mathcal{F} \hat\chi   =   M.
\end{equation}

For the specific illumination \eqref{eqInitialIllumination}, one has
\begin{equation}\label{incidentfieldfourier}
 \hat E^{(0)}(\omega,x)=\left(\int_{-\infty}^\infty f(t+\tfrac{x_3}c)\e^{\i\omega t}\d t\right)p = \hat f(\omega)\e^{-\i\frac\omega cx_3}p.
\end{equation}
Then, the operators $\mathcal{K}_0$ and $\mathcal M$ simplify to
\begin{equation}\label{eqOperatorKo}
(\mathcal{K}_0 v) (\omega, \rho \vartheta) = -\frac{\omega^2\e^{\i\frac\omega c\rho}}{4\pi \rho c^2} \hat f(\omega) \int_{\R^3}\vartheta\times\big(\vartheta\times(v(\omega,y) \,p)\big)\e^{-\i\frac\omega c\left<\vartheta +e_3,y\right>}\d y
\end{equation}
and, recalling that $M_{r,3}=0$ since $p_3 =0$ (the polarisation in the incident direction is zero), 
\begin{equation}\label{eqOperatorM}
(\mathcal{M} v) (r, x) = \left( - \frac{p_j}{2\pi}\int_{-\infty}^\infty v_j (\omega,x)  \hat f(-\omega) \e^{\i\frac\omega c(2r-x_3)}\d\omega \right)_{j=1}^{2}.
\end{equation}

The operator $\mathcal{K}_0$ is derived from the Born approximation taking into account the far field approximation for the solution of the Lippmann--Schwinger equation~\eqref{eqLippmannSchwinger}. But, one could also neglect \autoref{enAssumption1} and \autoref{enAssumption2} and use the operator $\mathcal{K}$ corresponding to equation \eqref{eqLippmannSchwinger}, that is,
\[
(\mathcal{K} v) (\omega, x) = \left(\frac{\omega^2}{c^2} \id +\grad\divx\right)\int_{\R^3}G (\omega, x-y) v(\omega,y)\hat E(\omega,y)\d y ,
\]
and considering the non-linear forward operator $\mathcal{F}= \mathcal{M} \mathcal{K}.$

The next section focuses on the solution of \eqref{eqInverseProblem}, considering the operators $\mathcal K_0$ and $\mathcal M,$  given by \eqref{eqOperatorKo} and \eqref{eqOperatorM}, respectively. The inversion of $\mathcal{F}$ is performed in two steps, first $\mathcal M$ is inverted and then $\mathcal K_0 .$

\section{The Inverse Scattering Problem}\label{seInverseProblem}

In optical coherence tomography, the susceptibility $\chi$ of the sample is imaged from the measurements $M_r(x)$, $r\in(-\infty,R)$, $x\in\mathcal D$. In a first step, it is shown that the measurements allow us to reconstruct the scattered field on the detector~$\mathcal{D},$ that is inverting the operator \eqref{eqOperatorM}.

The vertical distribution $f:\R\to\R$ should additionally satisfy (see \autoref{fgMeasurementSetup})
\begin{equation}\label{eqInitialIlluminationPulseLength} 
\supp f\subset(\tfrac Rc,\tfrac Rc+\tfrac{2\delta}c)\subset(\tfrac Rc,\tfrac dc)\quad\text{for some}\quad\delta>0 . 
\end{equation} 
This guarantees that the initial field $E^{(0)}$ (and also the magnetic field $H^{(0)}$) not interact with the sample or the mirror for $t\le0$ and neither contribute to the measurement at the detectors for $t\ge0$ as required by \eqref{eqInitialConditionMirror} and \eqref{eqNoSignalFromIncomingWave}. 

The condition that the length of the support of $E^{(0)}$ is at most $2\delta$ (the assumption that the support starts at $\frac Rc$ is only made to simplify the notation) is required for \autoref{thElectricFieldFromMeasurements}. It ensures that the formula \eqref{eqEffectiveMeasurementsForwardModel} for the measurement data $M_r(x)$, $x\in\mathcal D$, vanishes for values $r\ge R$ so that the integral on the right hand side of \eqref{eqElectricFieldFromMeasurements} is only over the interval $(-\infty,R)$ where measurement data are obtained (recall that measurements are only performed for positions $r<R$ of the mirror).

\begin{proposition}\label{thElectricFieldFromMeasurements}
Let $E^{(0)}$ be an initial illumination of the form~\eqref{eqInitialIllumination}  satisfying \eqref{eqInitialIlluminationPulseLength}. Then, the measurements $M_r$ from \autoref{thElectricFieldFromMeasurements1} imply for the electric field $E$:
\begin{equation}\label{eqElectricFieldFromMeasurements}
(\hat E_j-\hat E_j^{(0)})(\omega,x) \overline{\hat f(\omega)}p_j = -\frac2c\int_{-\infty}^R M_{r,j}(x)\e^{-\i\frac\omega c(2r-x_3)}\d r
\end{equation}
for all $j\in\{1,2,3\}$, $\omega\in\R$, and $x\in\mathcal D$.
\end{proposition}

\begin{proof}
Remark that the formula \eqref{eqMeasurementsPlaneWave1} can be extended to all $r\in\R$ by setting $M_{r,j}(x)=0$ for $r\ge R$. Indeed, from \eqref{eqInitialIlluminationPulseLength} it follows that $E(t,\cdot)=E^{(0)}(t,\cdot)$ for all $t<\frac\delta c$. Since $E$ is a solution of the linear wave equation with constant wave speed $c$ on the half space given by $x_3>R-\delta,$ the difference between $E$ and $E^{(0)}$ caused by the sample needs at least time $\tfrac{d-R+\delta}c$ to travel from the point at $x_3=R-\delta$ to the detector at $x_3=d$, so:
\[ E(t,x)=E^{(0)}(t,x)\quad\text{for all}\quad x\in\R^3\;\text{with}\; x_3=d\;\text{and}\; ct<2\delta+d-R. \]
This means, that the integrand vanishes for $t< \frac{2\delta + d-R}c$. In the case of  $t \ge \frac{2\delta + d-R}c,$ it holds  for $r \ge R$ that 
\[ ct+2r-d \ge 2\delta+d-R+2R-d = R+2\delta, \]
so that $f(t+\frac{2r-d}c)=0$ by the assumption \eqref{eqInitialIlluminationPulseLength} on the support of $f.$ Therefore, for $r\ge R$, always one of the factors in the integrand in \eqref{eqMeasurementsPlaneWave1} is zero which implies that $M_r(x)=0$ for $r\ge R$ and $x\in\mathcal D$.

Thus, equation \eqref{eqElectricFieldFromMeasurements1} holds for all $r \in \R$ and applying the inverse Fourier transform with respect to $r,$ using that $\hat f(-\omega)=\overline{\hat f(\omega)}$ because $f$ is real valued, yields
\[ \frac2c\int_{-\infty}^\infty M_{r,j}(x)\e^{-\i\frac{2\omega r}c}\d r = -p_j(\hat E_j-\hat E_j^{(0)})(\omega,x)\overline{\hat f(\omega)}\e^{-\i\frac\omega cx_3}, \]
which can equivalently be written as \eqref{eqElectricFieldFromMeasurements}.
\end{proof}

This means that one can calculate from the Fourier transform of the measurements $r\mapsto M_r(x)$ at some frequency $\omega$ the Fourier transform of the electric field at $\omega$ as long as the Fourier transform of the initial wave $E^{(0)}$ does not vanish at $\omega$, that is for $\hat f(\omega)\ne0$. Thus, under the \autoref{enAssumption1} and \autoref{enAssumption2}, equation \eqref{eqElectricFieldFromMeasurements} can be solved for the electric field $\hat E$. \autoref{thElectricFieldFromMeasurements} thus provides the inverse of the operator $\mathcal M$ defined by \eqref{eqOperatorM}. Now, the inversion of the operator $\mathcal K_0$ given by \eqref{eqOperatorKo} is performed considering the optical properties of the sample.

\begin{proposition}\label{thSusceptibilityFromMeasurements}
Let $E^{(0)}(t,x)$ be given by the form~\eqref{eqInitialIllumination} with $p_3=0$ and the additional assumption \eqref{eqInitialIlluminationPulseLength}. Then, for every $\omega\in\R\setminus\{0\}$ with $\hat f(\omega)\ne0$, the formula
\begin{equation}\label{eqSusceptibilityFromMeasurements}
p_j\big[\vartheta\times(\vartheta\times\tilde\chi(\omega,\tfrac\omega c(\vartheta+e_3))p)\big]_j \simeq \frac{8\pi\rho c}{\omega^2 |\hat f(\omega)|^2}\int_{-\infty}^R M_{r,j}(\rho\vartheta)\e^{-\i\frac\omega c(2r-\rho(\vartheta_3-1))}\d r
\end{equation}
holds for all $j\in\{1,2\}$, $\vartheta\in S^2_+ :=\{\eta\in S^2\mid\eta_3>0\}$, and $\rho=\frac d{\vartheta_3}$ (asymptotically for $\chi\to0$ and $\rho\to\infty$).

Here $\tilde\chi$ denotes the Fourier transform of $\chi$ with respect to time and space, that is
\begin{equation}\label{eqSpatialFourierTransformSusceptibility}
\tilde\chi(\omega,k) = \int_{-\infty}^\infty\int_{\R^3}\chi(t,x)\e^{-\i\left<k,x\right>}\e^{\i\omega t}\d x\d t =\int_{\R^3} \hat \chi(\omega ,x)\e^{-\i\left<k,x\right>} \d x .
\end{equation}
\end{proposition}
\begin{proof}
Because of \eqref{incidentfieldfourier}, the Fourier transform of the electric field $E - E^{(0)}$ can be approximated, using \eqref{eqFarField} with $E \simeq E^{(1)}$ (by \autoref{enAssumption1} and \autoref{enAssumption2}), by
\begin{align*}
(\hat E-\hat E^{(0)})(\omega,\rho\vartheta)&\simeq-\frac{\omega^2\hat f(\omega)\e^{\i\frac\omega c\rho}}{4\pi\rho c^2}\int_{\R^3}\e^{-\i\frac\omega c(\left<y,\vartheta\right>+y_3)}\vartheta\times(\vartheta\times\hat\chi(\omega,y)p)\d y .
\end{align*}
Then, applying \eqref{eqSpatialFourierTransformSusceptibility}, one obtains
\begin{equation}\label{eqelectricfieldcross}
(\hat E-\hat E^{(0)})(\omega,\rho\vartheta)\simeq-\frac{\omega^2\hat f(\omega)\e^{\i\frac\omega c\rho}}{4\pi\rho c^2}\vartheta\times(\vartheta\times\tilde\chi(\omega,\tfrac\omega c(\vartheta+e_3))p).
\end{equation}
From \eqref{eqElectricFieldFromMeasurements}, it is known that, for $\hat f(\omega)\ne0$ and $p_j\ne0,$
\[
(\hat E_j-\hat E_j^{(0)})(\omega,\rho \vartheta) = -\frac2{ p_j c \overline{\hat f(\omega)} }\int_{-\infty}^R M_{r,j}( \rho \vartheta)\e^{-\i\frac\omega c(2r-\rho \vartheta_3)}\d r.
\]
This identity together with \eqref{eqelectricfieldcross}, asymptotically for $\omega\ne0 ,$ yields the statement \eqref{eqSusceptibilityFromMeasurements}.
\end{proof}

To derive reconstruction formulas, \autoref{thSusceptibilityFromMeasurements} is used, which states that from the measurements $M_r$  (under the \autoref{enAssumption1} and \autoref{enAssumption2}) the expression
\begin{equation}\label{eqDataFromMeasurements}
p_j\big[\vartheta\times(\vartheta\times\tilde\chi(\omega,\tfrac\omega c(\vartheta+e_3))p)\big]_j,\quad j=1,2,
\end{equation}
can be calculated. Here,  $p\in\R^2\times\{0\}$ denotes the polarisation of the initial illumination~$E^{(0)}$, see \eqref{eqInitialIllumination},  and $\vartheta\in S^2_+$ is the direction from the origin (where the sample is located) to a detector.

\subsection{The Isotropic Case}\label{seIsotropicCase}
This section analyses the special case of an isotropic medium, meaning that the susceptibility matrix $\chi$ is just a multiple of the unit matrix, so in the following $\chi$ is identified with a scalar. 

Then, from the sum of the measurements $M_{r,1}$ and $M_{r,2},$ using the formula~\eqref{eqSusceptibilityFromMeasurements}, one obtains the expression
\[ \tilde\chi(\omega,\tfrac\omega c(\vartheta+e_3))\left<p,\vartheta\times(\vartheta\times p)\right> = \tilde\chi(\omega,\tfrac\omega c(\vartheta+e_3))(\left<\vartheta,p\right>^2-|p|^2). \]
Since $\left<\vartheta,p\right>^2<|p|^2$ for every combination of $p\in\R^2\times\{0\}$ and $\vartheta\in S^2_+$, one has direct access to the spatial and temporal Fourier transform
\begin{equation}\label{eqMeasurementDataIsotropic}
\tilde\chi(\omega,\tfrac\omega c(\vartheta+e_3)),\quad\omega\in\R\setminus\{0\},\;\vartheta\in S^2_+,
\end{equation}
of $\chi$ in a subset of $\R \times \R^3.$

However, it remains the problem of reconstructing the four dimensional susceptibility data~$\chi$ from the three dimensional measurement data \eqref{eqMeasurementDataIsotropic}. In the following, some different additional assumptions are discussed to compensate the lack of dimension, see \autoref{tablerecon}.

\begin{table}[ht]
\begin{center}
\begin{tabular}{|p{0.34\textwidth}|p{0.4\textwidth}|l|} \hline
{\bf Assumptions} & {\bf Reconstruction method} & {\bf Section} \\ \hline\hline
$\tilde\chi(\omega,k)=\tilde\chi(k)$ & Reconstruction from partial (three dimensional) Fourier data:\newline $\tilde\chi(k)$, $k\in\R^3$, $\measuredangle(k,e_3)\in(-\tfrac\pi4,\tfrac\pi4)$ & \ref{seIsotropicNoFrequency} \\ \hline
$\tilde\chi(\omega,k)=\tilde\chi(k_3)$ & Reconstruction from full (one dimensional) Fourier data: \newline$\tilde\chi(k_3)$, $k_3\in\R\setminus\{0\}$ & \ref{seIsotropicNoFrequencyOneDimensional}  \\ \hline
$\supp\chi(\cdot,x)\subset[0,T]$\newline$\mathcal R(\chi(\tau,\cdot))(\cdot,\varphi)$ is piecewise constant & Recursive formula to get limited angle Radon data \newline $\mathcal R(\chi(\tau,\cdot))(\sigma,\varphi)$, $\sigma\in\R$, $\varphi\in S^2$ \newline with $\measuredangle(\varphi,e_3)\in(-\frac\pi4,\frac\pi4)$ & \ref{seIsotropicDispersive} \\ \hline
$\chi(\tau,x)=\delta(x_1)\delta(x_2)\chi(\tau,x_3)$, \newline$\supp\chi(\cdot,x)\subset[0,T]$, and \newline $\chi(\tau,\cdot)$ is piecewise constant & Recursive formula to \newline reconstruct~$\chi$ & \ref{seIsotropicNoFrequencyOneDimensionalLayered} \\ \hline
\end{tabular}
\end{center}
\caption{Different assumptions about the susceptibility and the corresponding reconstruction formulas. Here, $(\mathcal Rg)(\sigma,\varphi)=\int_{\{x\in\R^3\mid\left<x,\varphi\right>=\sigma\}}g(y)\d s(y)$, $\sigma\in\R$, $\varphi\in S^2$, denotes the Radon transform of a function $g:\R^3\to\R$.}
\label{tablerecon}
\end{table}

\subsubsection{Non-dispersive Medium in Full Field OCT}\label{seIsotropicNoFrequency}
The model is simplified by assuming an immediate reaction of the sample to the exterior electric field in \eqref{eqSusceptibility}. This means that $\chi$ can be considered as a delta distribution in time so that its temporal Fourier transform $\hat\chi$ does not depend on frequency, that is $\hat \chi (\omega, x) = \hat \chi ( x).$ Thus, the reconstruction reduces to the problem of finding $\hat\chi$ from its partial (spatial) Fourier data
\begin{align*}
\tilde\chi(k)\quad\text{for}\quad k\in\{&\tfrac\omega c(\vartheta+e_3)\in\R^3\mid\vartheta\in S^2_+,\;\omega\in\R\setminus\{0\}\} \\
&=\{\kappa\in\R^3\setminus\{0\}\mid\arccos(\langle\tfrac\kappa{|\kappa|},e_3\rangle)\in(-\tfrac\pi4,\tfrac\pi4)\}.
\end{align*}

Thus, only the Fourier data of $\chi$ in the right circular cone $\mathcal C$ with axis along $e_3$ and aperture~$\frac\pi2$ are observed (see \autoref{fgFourierData}). In practice, these data are usually only available for a small range of frequencies $\omega$.

Inverse scattering for full field OCT, under the Born approximation, has been considered by Marks \textit{et al} \cite{MarDavBopCar09, MarRalBopCar07} where algorithms to recover the scalar susceptibility were proposed.

\begin{figure}[ht]
\begin{center}
\includegraphics{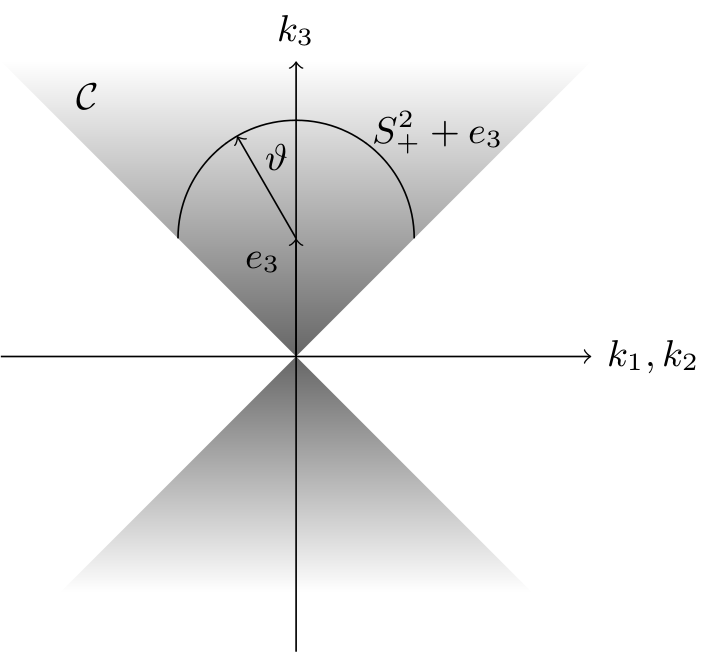}
\caption{Region $\mathcal C\cong\R \times(S^2_++e_3)$ of the available Fourier data of $\chi.$}
\label{fgFourierData}
\end{center}
\end{figure}

\subsubsection{Non-dispersive Medium with Focused Illumination}\label{seIsotropicNoFrequencyOneDimensional}
In standard OCT, the illumination is focused to a small region inside the object so that the function $\chi$ can be assumed to be constant in the directions $e_1$ and $e_2$ (locally the illumination is still assumed to be properly described by a plane wave). Then, the problem can be reduced by two dimensions assuming that the illumination is described by a delta distribution in these two directions. As before, $\chi$ is assumed to be frequency independent, so that
\[ \hat\chi(\omega,x)=\delta(x_1)\delta(x_2)\hat\chi(x_3), \]
this means that the spatial and temporal Fourier transform \eqref{eqSpatialFourierTransformSusceptibility} fulfils $\tilde \chi (\omega, k) = \tilde \chi (k_3).$  In this case, the two dimensional detector array can be replaced by a single point detector located at $de_3$.

Then, the measurement data \eqref{eqMeasurementDataIsotropic} in direction $\vartheta=e_3$ provide the Fourier transform
\[ \tilde\chi(\tfrac{2\omega}c) \quad\text{for all}\quad\omega\in\R\setminus\{0\}. \]

Therefore, the reconstruction of the (one dimensional) susceptibility $x_3\mapsto\hat\chi(x_3)$ can be simply obtained by an inverse Fourier transform.

This one dimensional analysis, has been used initially by Fercher \textit{et al} \cite{FerHitKamZai95}, reviewed in~\cite{Fer96} and by Hellmuth \cite{Hel96} to describe time domain OCT. Ralston \textit{et al} \cite{RalMarKamBop05, RalMarCarBop06} described the OCT system using a single backscattering model. The solution was given through numerical simulation using regularised least squares methods.

\subsubsection{Dispersive Medium}\label{seIsotropicDispersive}
However, in the case of a dispersive medium, that is frequency dependent $\hat \chi,$ the difficulty is to reconstruct the four dimensional function $\chi:\R\times\R^3\to\C$ from the three dimensional data
\begin{equation}\label{eqMeasurementsIsotropic}
\hat m:\R\times S_+^2\to\C,\quad\hat m(\omega,\vartheta) = \tilde\chi(\omega,\tfrac\omega c(\vartheta+e_3)).
\end{equation}

\begin{lemma}\label{thSusceptibilityParametrisation}
Let $\hat m$ be given by \eqref{eqMeasurementsIsotropic}. Then its inverse Fourier transform $m:\R\times S_+^2\to\C$ with respect to the first variable is given by
\begin{equation}\label{eqMeasurementsIsotropicDispersion}
m(t,\vartheta) = \frac c{\sqrt{2(1+\vartheta_3)}}\int_{-\infty}^\infty \bar\chi(\tau; \tau-t,\vartheta) \d\tau,\quad t\in\R,\;\vartheta\in S^2_+,
\end{equation}
where \[\bar\chi(\tau;\sigma,\vartheta)=\int_{E_{\sigma,\vartheta}}\chi(\tau,y)\d s(y), \quad \tau, \sigma  \in \R, \,\vartheta \in S^2_+ ,\]
and $E_{\sigma,\vartheta}$ denotes the plane
\begin{equation}\label{eqPlane}
E_{\sigma,\vartheta} = \{y\in\R^3\mid\left<\vartheta+e_3,y\right>=c\sigma\},\quad\sigma\in\R,\;\vartheta\in S^2_+.
\end{equation}
\end{lemma}

\begin{figure}[ht]
\begin{center}
\includegraphics{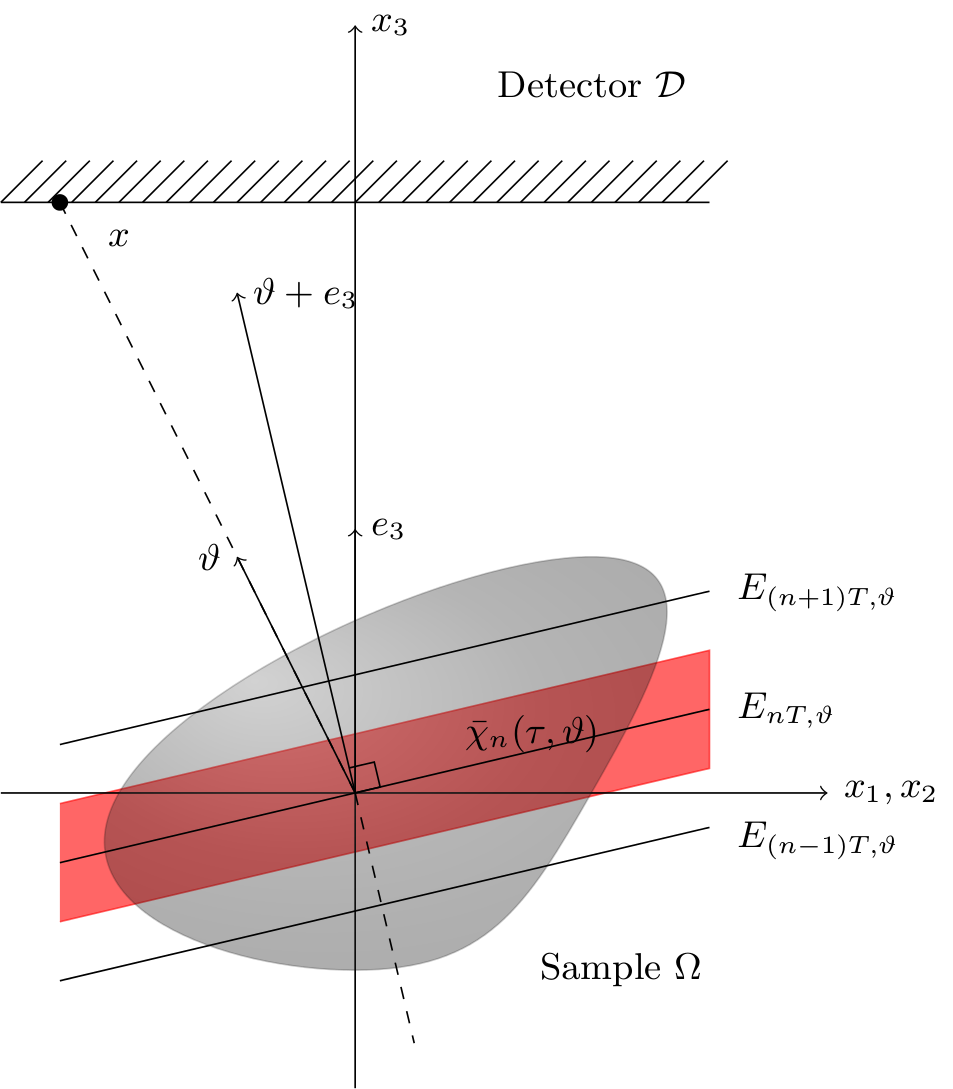}
\caption{Discretisation of $\chi$ with respect to the detection points.}
\label{fgDiscretisationRadonData}
\end{center}
\end{figure}

\begin{proof}
Taking the inverse temporal Fourier transform of $\hat m$ and using \eqref{eqSpatialFourierTransformSusceptibility}, it follows that
\begin{align*}
m(t,\vartheta) &= \frac1{2\pi}\int_{-\infty}^\infty\tilde\chi(\omega,\tfrac\omega c(\vartheta+e_3)) \e^{-\i\omega t} \d\omega \\
&= \frac1{2\pi}\int_{-\infty}^\infty\int_{\R^3}\hat\chi(\omega,x)\e^{-\i\frac\omega c\left<\vartheta+e_3,x\right>}\e^{-\i\omega t}\d x\d\omega.
\end{align*}
Interchanging the order of integration, the integral over $\omega$ is again described by an inverse Fourier transform and the previous equation becomes
\[ m(t,\vartheta) = \int_{\R^3}\chi(t+\tfrac1c\left<\vartheta+e_3,x\right>,x)\d x. \]

Substituting then the variable $x_3$ by $\tau=t+\frac1c\left<\vartheta+e_3,x\right>$, this can be written as
\begin{equation}\label{eqMeasurementsReparametrisation}
m(t,\vartheta) = \frac c{1+\vartheta_3}\int_{-\infty}^\infty\int_{-\infty}^\infty\int_{-\infty}^\infty\chi(\tau,\psi_{\tau-t,\vartheta}(x_1,x_2))\d x_1\d x_2\d\tau
\end{equation}
with the function
\[ \psi_{\sigma,\vartheta}:\R^2\to\R^3,\quad\psi_{\sigma,\vartheta}(x_1,x_2)=\left(x_1,x_2,\frac{c\sigma}{1+\vartheta_3}-\frac{\vartheta_1x_1+\vartheta_2x_2}{1+\vartheta_3}\right). \]
Now, $\psi_{\sigma,\vartheta}$ is seen to be the parametrisation of the plane
\[ E_{\sigma,\vartheta} = \{y\in\R^3\mid\left<v_\vartheta,y\right>=a_{\sigma,\vartheta}\}\quad\text{with}\quad v_\vartheta=\begin{pmatrix}\frac{\vartheta_1}{1+\vartheta_3}\\\frac{\vartheta_2}{1+\vartheta_3}\\1\end{pmatrix}\quad\text{and}\quad a_{\sigma,\vartheta}=\frac{c\sigma}{1+\vartheta_3}, \]
see \autoref{fgDiscretisationRadonData}.
The square root of the Gram determinant of the parametrisation $\psi_{\sigma,\vartheta}$ is now given by the length of the vector $v_\vartheta=\frac{\partial\psi_{\sigma,\vartheta}}{\partial x_1}\times\frac{\partial\psi_{\sigma,\vartheta}}{\partial x_2}$, which implies that 
\[ \int_{-\infty}^\infty\int_{-\infty}^\infty\chi(\tau,\psi_{\tau-t,\vartheta}(x_1,x_2))\d x_1\d x_2 = \sqrt{\frac{1+\vartheta_3}2}\int_{E_{\tau-t,\vartheta}}\chi(\tau,y)\d s(y). \]
Plugging this into \eqref{eqMeasurementsReparametrisation} yields the claim.
\end{proof}

Thus, the measurements give the combination \eqref{eqMeasurementsIsotropicDispersion} of values $\bar\chi$ of the Radon transform of the function $\chi(\tau,\cdot)$. It seems, however, impossible to recover the values $\bar\chi(\tau;\sigma,\vartheta)$ from this combination
\[ m(t,\vartheta) = \frac c{\sqrt{2(1+\vartheta_3)}}\int_{-\infty}^\infty\bar\chi(\tau;\tau-t,\vartheta)\d\tau, \]
since (for every fixed angle $\vartheta\in S^2_+$) one would have to reconstruct a function on $\R^2$ from one dimensional data.

To overcome this problem, the function $\bar\chi(\tau;\cdot,\vartheta)$ is going to be discretised for every $\tau\in\R$ and $\vartheta\in S^2_+$, where the step size will depend on the size of the support of $\chi(\cdot,x)$. 

Let us therefore consider the following assumption.
\begin{assumption}\label{eqSupportChi}
The support of $\chi$ in the time variable is contained in a small interval $[0,T]$ for some $T>0$:
\begin{equation*}
\supp\chi(\cdot,x)\subset[0,T]\quad\text{for all}\quad x\in\R^3.
\end{equation*}
\end{assumption}

Then, the following discretisation 
\[ \bar\chi_n(\tau,\vartheta) = \int_{E_{nT,\vartheta}}\chi(\tau,y)\d s(y),\quad n\in\Z,\;\tau\in(0,T),\;\vartheta\in S^2_+, \]
of the Radon transform of the functions $\chi(\tau,\cdot)$ is considered, where $E_{\sigma,\vartheta}$ denotes the plane defined in~\eqref{eqPlane}.

\begin{assumption}\label{eqDiscretisationAssumption}
The value $\bar\chi_n(\tau,\vartheta)$ is a good approximation for the integral of the function $\chi(\tau,\cdot)$ over the planes $E_{nT+\varepsilon,\vartheta}$ for all $\varepsilon\in[-\frac T2,\frac T2)$ (see \autoref{fgDiscretisationRadonData}), that is
\begin{equation*}
\bar\chi_n(\tau,\vartheta) \approx \int_{E_{nT+\varepsilon,\vartheta}}\chi(\tau,y)\d s(y),\quad \varepsilon\in[-\tfrac T2,\tfrac T2),\;n\in\Z,\;\tau\in(0,T),\;\vartheta\in S^2_+ .
\end{equation*}
\end{assumption}
Under the \autoref{eqDiscretisationAssumption}, equation \eqref{eqMeasurementsIsotropicDispersion} can be rewritten in the form
\[ m(t,\vartheta)\approx\frac c{\sqrt{2(1+\vartheta_3)}}\int_0^T\bar\chi_{N(\tau-t)}(\tau,\vartheta)\d\tau, \]
where $N(\sigma)=\left\lfloor\frac\sigma T+\frac12\right\rfloor$ denotes the integer closest to $\frac\sigma T$.
This (approximate) identity can now be iteratively solved for $\bar\chi$.

\begin{proposition}\label{reconstructionformula}
Let
\[ \bar m(t,\vartheta) = \frac c{\sqrt{2(1+\vartheta_3)}}\int_0^T\bar\chi_{N(\tau-t)}(\tau,\vartheta)\d\tau,\quad \vartheta\in S^2_+,\;t\in\R, \]
for some constant $T>0$ with the integer valued function $N(\sigma)=\left\lfloor\frac\sigma T+\frac12\right\rfloor$.

Then, $\bar\chi$ fulfils the recursion relation
\begin{equation}\label{eqRecursionRelationSusceptibility}
\bar\chi_n(\tau,\vartheta) = \bar\chi_{n+1}(\tau,\vartheta)+\frac{\sqrt{2(1+\vartheta_3)}}c\frac{\partial\bar m}{\partial t}(\tau-(n+\tfrac12)T,\vartheta),\quad n\in\Z,\;\tau\in(0,T),\;\vartheta\in S^2_+.
\end{equation}
\end{proposition}
\begin{proof}
Let $t=-nT+\varepsilon$ with $\varepsilon\in[-\frac T2,\frac T2)$ and $n\in\Z$, then
\[ N(\tau-t)=\begin{cases}n&\text{if}\;\tau\in(0,\tfrac T2+\varepsilon),\\n+1&\text{if}\;\tau\in[\tfrac T2+\varepsilon,T).\end{cases} \]
Taking an arbitrary $\varepsilon\in[-\frac T2,\frac T2)$ and $n\in\Z,$ $\bar m$ can be formulated as
\[ \bar m(-nT+\varepsilon,\vartheta) = \frac c{\sqrt{2(1+\vartheta_3)}}\left(\int_0^{\frac T2+\varepsilon}\bar\chi_n(\tau,\vartheta)\d\tau+\int_{\frac T2+\varepsilon}^T\bar\chi_{n+1}(\tau,\vartheta)\d\tau\right). \]
Differentiating this equation with respect to $\varepsilon$, it follows that
\[ \frac{\partial\bar m}{\partial t}(-nT+\varepsilon,\vartheta) = \frac c{\sqrt{2(1+\vartheta_3)}}\left(\bar\chi_n(\tfrac T2+\varepsilon,\vartheta)-\bar\chi_{n+1}(\tfrac T2+\varepsilon,\vartheta)\right), \]
which (with $\tau=\frac T2+\varepsilon$) is equivalent to \eqref{eqRecursionRelationSusceptibility}.
\end{proof}

Thus, given that $\bar\chi_n(t,\vartheta)=0$ for sufficiently large $n\in\Z$ (recall that $\supp\chi(\tau,\cdot)\subset\Omega$ for all $\tau\in(0,T)$), one can recursively reconstruct $\bar\chi$, to obtain the data
\[ \int_{E_{\sigma,\vartheta}}\chi(\tau,y)\d s(y)\quad\text{for all}\quad\tau\in[0,T),\;\sigma\in\R,\;\vartheta\in S^2_+ \]
for the Radon transform of $\chi(\tau,\cdot)$. 

However, since the plane $E_{\sigma,\vartheta}$ is by its definition \eqref{eqPlane} orthogonal to the vector $\vartheta+e_3$ for $\vartheta\in S^2_+$, this provides only the values of the Radon transform corresponding to planes which are orthogonal to a vector in the cone $\mathcal C$, see \autoref{fgFourierData}. For the reconstruction, one therefore still has to invert a limited angle Radon transform.

\subsubsection{Dispersive Layered Medium with Focused Illumination}\label{seIsotropicNoFrequencyOneDimensionalLayered}

Except from ophthalmology, OCT is also widely used for investigation of
skin deceases, such as cancer. From the mathematical point of view, this simplifies the main model since the human skin can be described as a multi-layer structure with different optical properties and varying thicknesses in each layer. 

As in \autoref{seIsotropicNoFrequencyOneDimensional}, the incident field is considered to propagate with normal incidence to the interface $x_3 =L$ and the detector array is replaced by a single point detector located at $d e_3.$ The susceptibility is simplified as
\[ \chi(t,x)=\delta(x_1)\delta(x_2) \chi(t, x_3), \]
and therefore the measurements provide the data, see \eqref{eqMeasurementsIsotropicDispersion} with  $\tilde\chi (\omega, k) = \tilde\chi (\omega, k_3),$
\begin{equation*}\label{datafreqlayered}
\hat m (\omega) = \tilde\chi(\omega,\tfrac\omega c 2e_3),\quad\omega\in\R\setminus\{0\} .
\end{equation*}

Considering the special structure of a layered medium, the susceptibility is described by a piecewise  constant function in $x_3$. This means explicitly that $\chi$ has the form
\begin{equation}\label{layeredmediumsusceptibility}
\chi (t, x_3) = \left\{
     \begin{array}{ll}
       \chi_0 :=0, &  x_3 \notin [0,L] \\
       \chi_n (t), & x_3 \in [L_n ,L_{n+1})
     \end{array}
   \right., \quad n=1,\ldots,N
\end{equation}
with (unknown) parameters $L=L_1 > L_2 > \ldots > L_{N+1} = 0$ characterising the thicknesses of the $N$ layers and (unknown) functions $\chi_n.$

\autoref{thSusceptibilityParametrisation}, for $\vartheta = e_3,$ gives
\[ m(t)=\frac c2\int_{-\infty}^\infty \bar\chi (\tau;\tau-t)\d\tau, \quad \mbox{where} \quad \bar \chi (\tau; \sigma)= \chi (\tau, \tfrac{c\sigma}2).\] 

Remarking that $\bar\chi$ is piecewise constant \eqref{layeredmediumsusceptibility} and additionally assuming that $\chi (\cdot, x_3)$ has compact support, see \autoref{eqSupportChi}, with $T< \tfrac2c \min_n (L_n - L_{n+1})$ \autoref{reconstructionformula} can be applied for $\vartheta =e_3$ to iteratively reconstruct $\chi$ starting from $\chi_0 =0.$

\paragraph{Modified Born Approximation \\}
In the proposed iteration scheme, \autoref{reconstructionformula}, the travelling of the incident field through the sample before reaching a ``specific'' layer, where the susceptibility is to be reconstructed, is not considered. To do so, a modified iteration method is presented describing the travelling of the light through the different layers using Frensel's equations.

The main idea is to consider, for example, in the second step of the recursive formula, given $\chi_1$ to find $\chi_2 ,$ as incident the field $\hat E^{(0)}, $  given by (\ref{incidentfieldfourier}), travelled also through the first layer. This process can be continued to the next steps. 

Let us first introduce some notations which will be used in the following. The fields $\hat E_n^{(r)}$ and $\hat E_n^{(t)}$ denote the reflected and the transmitted fields, with respect to the boundary $L_n,$ respectively. The transmitted field $\hat E_n^{(t)}$ after travelling through the $n$-th layer is incident on the $L_{n+1}$ boundary and is denoted by $\hat E_{n+1}^{(0)}.$ The reflected field by the $L_{n+1}$ boundary back to the $L_{n}$ boundary will be denoted by $\hat E_{n+1}^{(r)}$ and by $\hat E_{n}^{(r')}$ after travelling through the $n$-th layer (see  \autoref{fgLayered}). To simplify this model, multiple reflections are not included and the electric fields are taken to be tangential to the interface planes.

\begin{figure}[ht]
\begin{center}
 \includegraphics{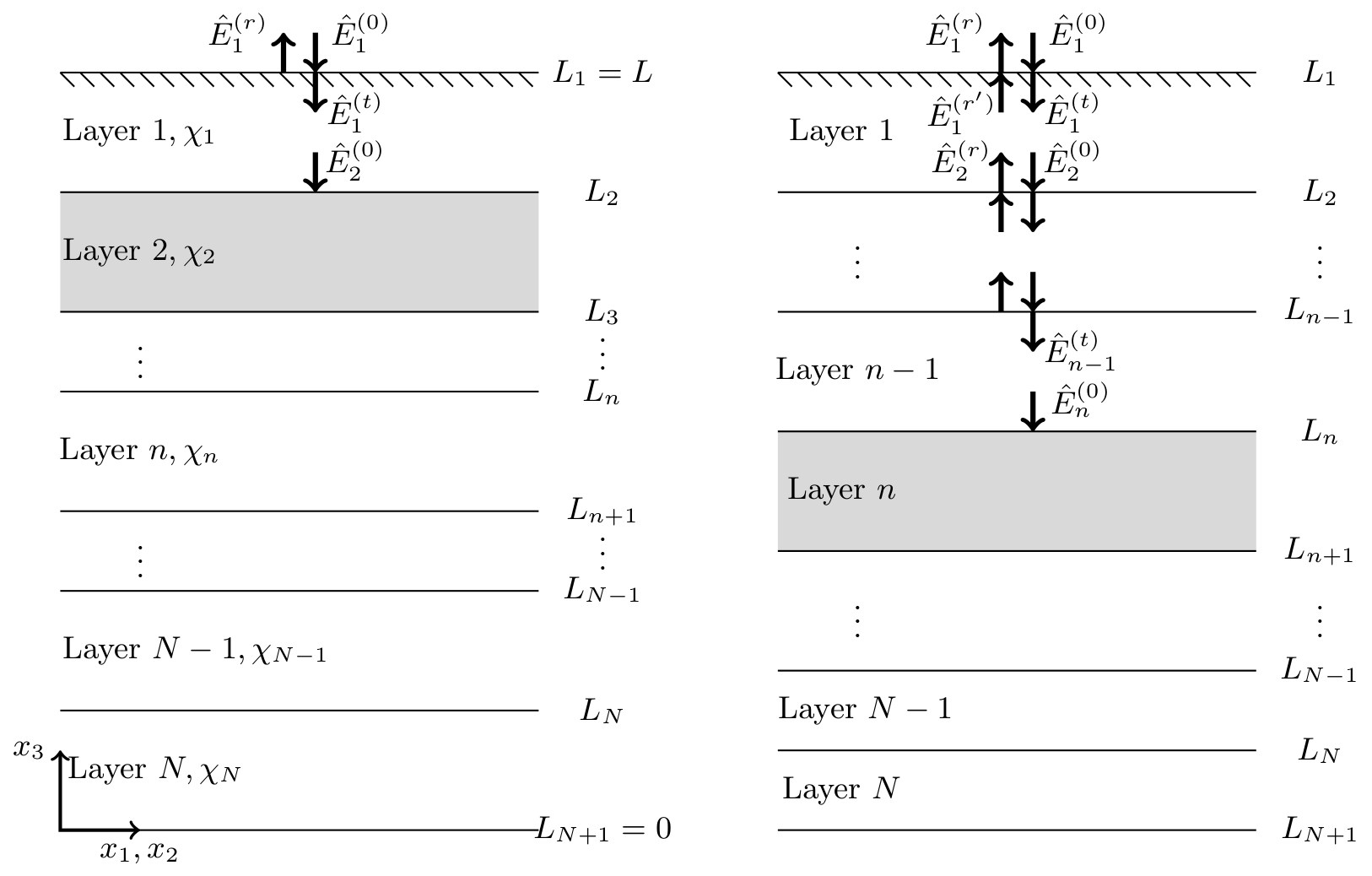}
\caption{Layered medium. Propagation of the initial field through the sample incident on the second layer (left image) and in general on the $n$th layer, for $n=3,...,N$ (right image).}
\label{fgLayered}
\end{center}
\end{figure}

\begin{lemma}
Let the sample have susceptibility given by \eqref{layeredmediumsusceptibility} and let $\rho_n$ and $\tau_n$ denote the reflection and the corresponding transmission coefficients for the $L_n$ boundary, respectively. Then, the field incident on the $n$th layer with respect to the initial incident field $\hat E^{(0)} := \hat E_1^{(0)}$ is given by
\begin{equation*}
\begin{pmatrix} \hat E_{n}^{(0)}\\ 0 \end{pmatrix} = \left( \mathcal{M}_1 \cdot \mathcal{M}_2 \cdot ... \cdot \mathcal{M}_{n-1} \right)^{-1}  \begin{pmatrix} \hat E^{(0)}\\ \hat E_1^{(r)} \end{pmatrix} \quad \mbox{for} \quad n=N-1,...,2
\end{equation*}
assuming no backward field in the $n$-th layer, where
\[
\mathcal{M}_n = \frac{1}{\tau_n} \begin{pmatrix}
 \e^{\i \tfrac{\omega}{c}\sqrt{\chi_n +1} (L_n - L_{n+1}) } & \rho_n \e^{-\i \tfrac{\omega}{c}\sqrt{\chi_n +1} (L_n - L_{n+1}) }  \\
 \rho_n \e^{\i \tfrac{\omega}{c}\sqrt{\chi_n +1} (L_n - L_{n+1}) } & \e^{-\i \tfrac{\omega}{c}\sqrt{\chi_n +1} (L_n - L_{n+1}) } 
 \end{pmatrix}.
\]
\end{lemma}

\begin{proof}
Because of the assumptions (normal incidence, $\hat E^{(0)}$ tangential to the boundary) the boundary conditions require the continuity of the total (upward and downward) electric and magnetic fields. Then, the reflection $\rho_n$ and the corresponding transmission $\tau_n$ coefficients for the $L_n$ boundary in terms of the susceptibility are given by \cite{Hec02}
\begin{equation*}
\rho_n= \frac{\sqrt{\chi_{n-1} +1}- \sqrt{\chi_{n} +1}}{\sqrt{\chi_{n-1} +1}+ \sqrt{\chi_n +1}}, \quad \tau_n= 1+ \rho_n .
\end{equation*}
To determine the propagation equations for the electric fields, the transfer matrices formulation is applied \cite{Orf02}. In particular, the fields at the top of the $n$th layer can 
be computed with respect to the fields at the top of the $(n+1)$th using
\begin{equation*}
\begin{pmatrix} \hat E_n^{(0)} \\ \hat E_n^{(r)} \end{pmatrix} = \mathcal{M}_n \begin{pmatrix} \hat E_{n+1}^{(0)}\\ \hat E_{n+1}^{(r)} \end{pmatrix} \quad \mbox{for} \quad n=N-1,...,1.
\end{equation*}
and with respect to the incident field,
\begin{equation*}
 \begin{pmatrix} \hat E_1^{(0)}\\ \hat E_1^{(r)} \end{pmatrix} = \mathcal{M}_1 \cdot ... \cdot \mathcal{M}_{n-1} \begin{pmatrix} \hat E_n^{(0)} \\ \hat E_n^{(r)} \end{pmatrix} \quad \mbox{for} \quad n=N-1,...,2.
\end{equation*}
\end{proof}

From the previous result, given $\chi_n $ (by the recursion relation of \autoref{reconstructionformula}), the matrix $\mathcal M_{n+1}$ is computed to obtain the update $\hat E_{n+1}^{(0)}$ which is then incident to the rest part of the sample. This means, that $\hat E^{(0)}$ is replaced by $\hat E_{n+1}^{(0)}$  in the derivation of the measurements and the recursion relation (\autoref{thSusceptibilityParametrisation} and \autoref{reconstructionformula}) for computing $\chi_{n+1}.$ For example, in the second step to reconstruct $\chi_2 ,$ the incident field is simply given by 
\begin{equation*}
\hat E_2^{(0)} = \tau_1 \\e^{-\i \tfrac{\omega}c \sqrt{\chi_1 +1} (L_1 - L_2 ) } \hat E^{(0)}.
\end{equation*} 

The only unknown in this representation is the boundary $L_2$ which can be approximated considering the point where change in the value of the measured function $\bar m$ is observed. The following analysis can be also extended for anisotropic media, but in a more complicated context since the displacement $D$ and the electric field $E$ are not always parallel.

A simplification usually made here is to consider the sample field as the sum of all the discrete reflections and neglect dispersion. This mathematical model was adopted by Bruno and Chaubell \cite{BruCha05} for solving the inverse scattering problem of determining the refractive index and the width of each layer from the output data. The solution was obtained using the Gauss--Newton method and the effect of the initial guesses was also considered.

Concluding, the travelling of the scattered field from the $n$-th layer through the sample could also be considered. Since the spherical waves can be represented as a superposition of plane waves by using similar techniques, in a more complicated form, one can obtain the transmitted scattered field.

\subsection{The Anisotropic Case}
In the anisotropic case, the susceptibility $\chi$ cannot be considered a multiple of the identity. Therefore, the problem is to reconstruct from the expressions
\[ p_j\big[\vartheta\times(\vartheta\times\tilde\chi(\omega,\tfrac\omega c(\vartheta+e_3))p)\big]_j,\quad j=1,2, \]
see \eqref{eqDataFromMeasurements}, the matrix valued function $\chi:\R\times\R^3\to\R^{3\times3}$, where it is assumed that measurements for every polarisation $p\in\R^2\times\{0\}$ of the initial field $E^{(0)}$ are available.

Introducing in analogy to \eqref{eqMeasurementsIsotropic} the function
\[ \hat m_{p,j}:\R\times S^2_+\to\C,\quad \hat m_{p,j}(\omega,\vartheta)=\tilde\chi_{\vartheta,p,j}(\omega,\tfrac\omega c(\vartheta+e_3)), \]
where $\tilde\chi_{\vartheta,p,j}$ is for every $\vartheta\in S^2_+$, $p\in\R^2\times\{0\}$, and $j\in\{1,2\}$ the (spatial and temporal) Fourier transform of 
\[ \chi_{\vartheta,p,j}:\R\times\R^3\to\R,\quad\chi_{\vartheta,p,j}(t,x) = p_j\big[\vartheta\times(\vartheta\times\chi(t,x)p)\big]_j, \]
\autoref{thSusceptibilityParametrisation} (with $m$ replaced by $m_{p,j}$ and $\chi$ replaced by $\chi_{\vartheta,p,j}$) can be applied to find that the inverse Fourier transform of $\hat m_{p,j}$ with respect to its first variable fulfils
\[ m_{p,j}(t,\vartheta) = \frac c{\sqrt{2(1+\vartheta_3)}}\int_{-\infty}^\infty\int_{E_{\tau-t,\vartheta}}\chi_{\vartheta,p,j}(\tau,y)\d s(y)\d\tau. \]

Now, the same assumptions as in \autoref{seIsotropicDispersive} are considered, namely \autoref{eqSupportChi} and similar to \autoref{eqDiscretisationAssumption}:
\begin{assumption}
The approximation
\[ \int_{E_{nT,\vartheta}}\chi_{\vartheta,p,j}(\tau,y)\d s(y)\approx\int_{E_{nT+\varepsilon,\vartheta}}\chi_{\vartheta,p,j}(\tau,y)\d s(y)\quad\text{for all}\quad\varepsilon\in[-\tfrac T2,\tfrac T2)  \]
is for every $\tau\in\R$, $\vartheta\in S^2_+$, $n\in\Z$, $p\in\R^2\times\{0\}$, and $j\in\{1,2\}$ justified.
\end{assumption}

Then, \autoref{reconstructionformula} provides an approximate reconstruction formula for the functions
\begin{equation}\label{eqReconstructedMeasurementData}
\bar\chi_{p,j}(\tau;\sigma,\vartheta) = \int_{E_{\sigma,\vartheta}}\chi_{\vartheta,p,j}(\tau,y)\d s(y) = p_j\left[\vartheta\times\left(\vartheta\times\bar\chi(\tau;\sigma,\vartheta)p\right)\right]_j
\end{equation}
for all $p\in\R^2\times\{0\}$, $\tau\in\R$, $\sigma\in\R$, $\vartheta\in S^2_+$, and $j\in\{1,2\}$, where
\begin{equation}\label{eqRadonTransformSuszeptibility}
\bar\chi(\tau;\sigma,\vartheta) = \int_{E_{\sigma,\vartheta}}\chi(\tau,y)\d s(y)
\end{equation}
denotes the two dimensional Radon transform data of the function $\chi(\tau,\cdot)$.% as in \autoref{seIsotropicDispersive}.

\begin{proposition}\label{thReconstructableMatrixEntries}
Let $\vartheta\in S^2_+$ be fixed %, $A\in\R^{3\times 3}$ and
%\[ a_{p,j} = p_j[\vartheta\times(\vartheta\times Ap)]_j. \]
and $a_{p,j}$, $p\in\R^2\times\{0\}$, $j=1,2$, be such that the equations
\begin{equation}\label{eqMeasurementsPolarisation}
p_j[\vartheta\times(\vartheta\times Xp)]_j = a_{p,j}\quad\text{for all}\quad p\in\R^2\times\{0\},\;j\in\{1,2\},
\end{equation}
for the matrix $X\in\R^{3\times3}$ have a solution.

Then $X\in\R^{3\times3}$ is a solution of \eqref{eqMeasurementsPolarisation} if and only if
\begin{equation}\label{eqReconstructableMatrixEntries}
(P_\vartheta X)_{k\ell}=B_{k\ell},\quad B=\begin{pmatrix}-a_{e_1,1}&a_{e_1,1}-a_{e_1+e_2,1}\\a_{e_2,2}-a_{e_1+e_2,2}&-a_{e_2,2}\end{pmatrix},\quad k,\ell\in\{1,2\},
\end{equation}
where $P_\vartheta\in\R^{3\times3}$ denotes the orthogonal projection in direction $\vartheta$.
\end{proposition}

\begin{proof}
First, remark that the equation system \eqref{eqMeasurementsPolarisation} is equivalent to the four equations
\begin{equation}\label{eqMeasurementsPolarisationReduced}
\begin{aligned}
a_{e_1,1} &= [\vartheta\times(\vartheta\times Xe_1)]_1,\qquad &a_{e_1+e_2,1} &= a_{e_1,1}+[\vartheta\times(\vartheta\times Xe_2)]_1,\\
a_{e_2,2} &= [\vartheta\times(\vartheta\times Xe_2)]_2,\qquad &a_{e_1+e_2,2} &= a_{e_2,2}+[\vartheta\times(\vartheta\times Xe_1)]_2,
\end{aligned}
\end{equation}
which correspond to the equations \eqref{eqMeasurementsPolarisation} for $(p,j)\in\{(e_1,1),(e_2,2),(e_1+e_2,1),(e_1+e_2,2)\}$. Indeed, for arbitrary polarisation $p=p_1e_1+p_2e_2$, the expression $p_j[\vartheta\times(\vartheta\times Xp)]_j$ can be written as a linear combination of the four expressions $[\vartheta\times(\vartheta\times Xe_i)]_k$, $i,k=1,2$:
\begin{align*}
p_1[\vartheta\times(\vartheta\times Xp)]_1 &= p_1^2[\vartheta\times(\vartheta\times Xe_1)]_1+p_1p_2[\vartheta\times(\vartheta\times Xe_2)]_1, \\
p_2[\vartheta\times(\vartheta\times Xp)]_2 &= p_1p_2[\vartheta\times(\vartheta\times Xe_1)]_2+p_2^2[\vartheta\times(\vartheta\times Xe_2)]_2,
\end{align*}
and is thus determined by \eqref{eqMeasurementsPolarisationReduced}.

Now, the equation system \eqref{eqMeasurementsPolarisationReduced} written in matrix form reads
\begin{equation}\label{eqReconstructableMatrixEntriesInComponents}
[\vartheta\times(\vartheta\times Xp)]_k = -\left[B\begin{pmatrix}p_1\\p_2\end{pmatrix}\right]_k,\quad k\in\{1,2\},%\begin{pmatrix}(p_1-p_2)a_{e_1,1}+p_2a_{e_1+e_2,1}\\(p_2-p_1)a_{e_2,2}+p_1a_{e_1+e_2,2}\end{pmatrix}
\end{equation}
for all $p\in\R^2\times\{0\}$ with $B$ defined by \eqref{eqReconstructableMatrixEntries}.

Decomposing $Xp=\left<\vartheta,Xp\right>\vartheta+P_\vartheta Xp$, where $P_\vartheta\in\R^{3\times3}$ denotes the orthogonal projection in direction $\vartheta$, and using that
\[ \vartheta\times(\vartheta\times Xp) = \vartheta\times(\vartheta\times P_\vartheta Xp) = \left<\vartheta,P_\vartheta Xp\right>\vartheta-P_\vartheta Xp = -P_\vartheta Xp, \] 
the equation \eqref{eqReconstructableMatrixEntriesInComponents} can be written in the form \eqref{eqReconstructableMatrixEntries}.
\end{proof}

\autoref{thReconstructableMatrixEntries} applied to the equations \eqref{eqReconstructedMeasurementData} for the matrix $X=\bar\chi(\tau;\sigma,\vartheta)$ for some fixed values $\tau, \sigma\in\R$ and $\vartheta\in S^2_+$ shows that the data $a_{p,j}=p_j[\vartheta\times(\vartheta\times\bar\chi(\tau;\sigma,\vartheta))]_j$ for $j=1,2$ and the three different polarisation vectors $p=e_1$, $p=e_2$, and $p=e_1+e_2$ uniquely determine with equation \eqref{eqReconstructableMatrixEntries} the projection
\[ (P_\vartheta\bar\chi(\tau;\sigma,\vartheta))_{k,\ell}=\int_{E_{\sigma,\vartheta}}(P_\vartheta\chi(\tau,y))_{k,\ell}\,\d s(y)\quad\text{for}\quad k,\ell\in\{1,2\}. \]
Moreover, measurements for additional polarisations $p$ do not provide any further informations so that at every detector point, corresponding to a direction $\vartheta\in S^2_+$, only the four elements $(P_\vartheta\chi)_{k,\ell}$, $k,\ell=1,2$, of the projection $P_\vartheta\chi$ influence the measurements.

To obtain additional data which make a full reconstruction of $\chi$ possible, one can carry out extra measurements after slight rotations of the sample.

So, let $R\in\mathrm{SO}(3)$ describe the rotation of the sample. Then the transformed susceptibility $\chi_R$ is given by 
\begin{equation}\label{eqRotatedSusceptibility}
\chi_R(t,y) = R\chi(t,R^{\mathrm T}y)R^{\mathrm T}.
\end{equation}

\begin{lemma}
Let $\chi:\R\times\R^3\to\R^{3\times3}$ be the susceptibility of the sample and $\vartheta\in S^2_+$ be given. Furthermore, let $R\in\mathrm{SO}(3)$ be such that there exists a constant $\alpha_R>0$ and a direction $\vartheta_R\in S^2_+$ with
\begin{equation}\label{eqRotatedDetector}
\vartheta_R+e_3=\alpha_RR(\vartheta+e_3)
\end{equation}
and define the susceptibility $\chi_R$ of the rotated sample by \eqref{eqRotatedSusceptibility}.

Then, the data
\begin{equation}\label{eqRotatedMeasurements}
\bar\chi_{R,p,j}(\tau;\sigma,\vartheta_R)=p_j\left[\vartheta_R\times\left(\vartheta_R\times\int_{E_{\sigma,\vartheta_R}}\chi_R(\tau,y)\d s(y)\right)\right]_j,
\end{equation}
%$\bar\chi_R(\tau;\sigma,\vartheta_R)=\int_{E_{\sigma,\vartheta_R}}\chi_R(\tau,y)\d s(y)$
corresponding to the measurements of the rotated sample at the detector in direction $\vartheta_R$, see~\eqref{eqReconstructedMeasurementData}, fulfil that
\begin{equation}\label{eqRotatedMeasurementsSimplified}
\bar\chi_{R,p,j}(\tau;\alpha_R\sigma,\vartheta_R) = p_j[\vartheta_R\times(\vartheta_R\times R\bar\chi(\tau;\sigma,\vartheta)R^{\mathrm T}]_j
\end{equation}
for all $\tau,\sigma\in\R$, $p\in\R^2\times\{0\}$, $j=1,2$, where $\bar\chi$ is given by \eqref{eqRadonTransformSuszeptibility}.
\end{lemma}
\begin{proof}
Inserting definition \eqref{eqRotatedSusceptibility} and substituting $z=R^{\mathrm T}y$, formula \eqref{eqRotatedMeasurements} becomes
\[ \bar\chi_{R,p,j}(\tau;\sigma,\vartheta_R)=p_j\left[\vartheta_R\times\left(\vartheta_R\times\int_{R^{\mathrm T}E_{\sigma,\vartheta_R}}R\chi(\tau,z)R^{\mathrm T}\d s(z)\right)\right]_j. \]
Since now, by the definition \eqref{eqPlane} of the plane $E_{\sigma,\vartheta}$ and by the definition \eqref{eqRotatedDetector} of $\vartheta_R$,
\begin{align*}
R^{\mathrm T}E_{\sigma,\vartheta_R}&=\{R^{\mathrm T}y\in\R^3\mid\left<\vartheta_R+e_3,y\right>=c\sigma\} \\
&= \{z\in\R^3\mid\left<R^{\mathrm T}(\vartheta_R+e_3),z\right>=c\sigma\} \\
&= \{z\in\R^3\mid\alpha_R\left<\vartheta+e_3,z\right>=c\sigma\} = E_{\frac\sigma{\alpha_R},\vartheta},
\end{align*}
it follows \eqref{eqRotatedMeasurementsSimplified}.
\end{proof}

This means that the data $\bar\chi_{R,p,j}(\tau;\alpha_R\sigma,\vartheta_R)$ obtained from a detector placed in the direction $\vartheta_R$, defined by \eqref{eqRotatedDetector}, depends only on the Radon transform data $\bar\chi(\tau;\sigma,\vartheta)$. However, it still remains the algebraic problem of solving the equations \eqref{eqRotatedMeasurementsSimplified} for different rotations $R$ to obtain the matrix $\bar\chi(\tau;\sigma,\vartheta)\in\R^{3\times3}$. 

\begin{proposition}\label{thReconstructionAnisotropic}
Let $A\in\R^{3\times 3}$ and $\vartheta\in S^2_+$ be given. Moreover, let $R_0,R_1,R_2\in\mathrm{SO}(3)$ be rotations so that every proper subset of $\{R^{\mathrm T}_0e_3,R^{\mathrm T}_1e_3,R^{\mathrm T}_2e_3,\vartheta+e_3\}$ is linearly independent and such that 
there exist for every $R\in\{R_0,R_1,R_2\}$ constants $\alpha_R>0$ and $\vartheta_R\in S^2_+$ fulfilling~\eqref{eqRotatedDetector}.

Let further $P\in\R^{2\times3}$ be the orthogonal projection in direction $e_3$, $P_\theta\in\R^{3\times3}$ the orthogonal projection in direction $\theta\in\R^3$, and
\[ B_R=\begin{pmatrix}-a_{R,e_1,1}&a_{R,e_1,1}-a_{R,e_1+e_2,1}\\a_{R,e_2,2}-a_{R,e_1+e_2,2}&-a_{R,e_2,2}\end{pmatrix},\quad a_{R,p,j} = p_j[\vartheta_R\times(\vartheta_R\times RAR^{\mathrm T}p)]_j, \]
for every $R\in\{R_0,R_1,R_2\}$.

Then, the equations
\begin{equation}\label{eqReconstructionAnisotropic} 
PP_{\vartheta_R}RXR^{\mathrm T}P^{\mathrm T}=B_R,\quad R\in\{R_0,R_1,R_2\}
\end{equation}
have the unique solution $X=A$.
\end{proposition}
\begin{proof}
Using that $\vartheta_R=\alpha_RR(\vartheta+e_3)-e_3$, see \eqref{eqRotatedDetector}, it follows with $Pe_3=0$ that
\[ PP_{\vartheta_R}=P(\mathbbm1-\vartheta_R\vartheta_R^{\mathrm T}) = P-\alpha_RPR(\vartheta+e_3)\vartheta_R^{\mathrm T}. \]
With this identity, the  equations \eqref{eqReconstructionAnisotropic} can be written in the form
\begin{equation}\label{eqProjectionIdentityForX}
PR(X-\alpha_R(\vartheta+e_3)\vartheta_R^{\mathrm T}RX)R^{\mathrm T}P^{\mathrm T} = B_R.
\end{equation}

Let now $\eta_R\in\R^3$ denote a unit vector orthogonal to $\vartheta+e_3$ and orthogonal to $R^{\mathrm T}e_3$. 
Then $R\eta_R$ is orthogonal to $e_3$ and therefore, with $P^{\mathrm T}P=P_{e_3}$,
\[ (PR\eta_R)^{\mathrm T}PR = (P_{e_3}R\eta_R)^{\mathrm T}R = (R\eta_R)^{\mathrm T}R = \eta_R^{\mathrm T}. \]
Thus, multiplying the equation \eqref{eqProjectionIdentityForX} from the left with $(PR\eta_R)^{\mathrm T}$, it follows that
\[ \eta_R^{\mathrm T}(X-\alpha_R(\vartheta+e_3)\vartheta_R^{\mathrm T}RX)R^{\mathrm T}P^{\mathrm T} = (PR\eta_R)^{\mathrm T}B_R. \]
Since now $\eta_R$ is orthogonal to $\vartheta+e_3$ this simplifies to
\begin{equation}\label{eqProjectionOfSusceptibility}
\eta_R^{\mathrm T}XR^{\mathrm T}P^{\mathrm T} = (PR\eta_R)^{\mathrm T}B_R.
\end{equation}

Remarking that the orthogonal projection onto the line $\R\eta_R$ can be written as the composition of two orthogonal projections onto orthogonal planes with intersection $\R\eta_R$:
\[ \eta_R\eta_R^{\mathrm T}=P_{\nu_R}P_{\vartheta+e_3}, \]
where $\nu_R$ is a unit vector orthogonal to $\eta_R$ and $\vartheta+e_3$, and multiplying equation \eqref{eqProjectionOfSusceptibility} from the left with $\eta_R$, one finds that
\[ P_{\nu_R}(P_{\vartheta+e_3}X)R^{\mathrm T}P^{\mathrm T} = \eta_R(PR\eta_R)^{\mathrm T}B_R. \]
Applying then the projection $P$ from the right and using that $R^{\mathrm T}P^{\mathrm T}P=R^{\mathrm T}P_{e_3}=P_{R^{\mathrm T}e_3}R^{\mathrm T}$, this can be written in the form
\begin{equation}\label{eqProjectionOfSusceptibilityOntoBlocks}
P_{\nu_R}(P_{\vartheta+e_3}X)P_{R^{\mathrm T}e_3} = \eta_R(PR\eta_R)^{\mathrm T}B_RPR.
\end{equation}
Evaluating this equation now for $R=R_0,R_1,R_2$ and remarking that $\{\nu_{R_0},\nu_{R_1},\nu_{R_2}\}$ and $\{R_0^{\mathrm T}e_3,R_1^{\mathrm T}e_3,R_2^{\mathrm T}e_3\}$ are linearly independent, one concludes that the $3\times3$ matrix $P_{\vartheta+e_3}X$ is uniquely determined by~\eqref{eqProjectionOfSusceptibilityOntoBlocks}.

However, it is also possible to calculate $X$ (and not only its projection $P_{\vartheta+e_3}X )$ from equation \eqref{eqProjectionIdentityForX}. Because
\[ X=P_{\vartheta+e_3}X+\frac1{|\vartheta+e_3|^2}(\vartheta+e_3)(\vartheta+e_3)^{\mathrm T}X, \]
it follows from \eqref{eqProjectionIdentityForX} that
\begin{multline}\label{eqPRequation}
PR\left(\frac{\vartheta+e_3}{|\vartheta+e_3|^2}(1-\alpha_R\vartheta_R^{\mathrm T}R(\vartheta+e_3))(\vartheta+e_3)^{\mathrm T}X\right)R^{\mathrm T}P^{\mathrm T} \\
= B_R-PR(\mathbbm 1-\alpha_R(\vartheta+e_3)\vartheta_R^{\mathrm T}R)P_{\vartheta+e_3}XR^{\mathrm T}P^{\mathrm T}.
\end{multline}
Plugging in the identity
\[ \alpha_R\vartheta_R^{\mathrm T}R(\vartheta+e_3)=\vartheta_R^{\mathrm T}(\vartheta_R+e_3)=1+\vartheta_{R,3}, \]
%and 
%\[ \vartheta_R^{\mathrm T}RP_{\vartheta+e_3}=(\alpha_R(\vartheta+e_3)^{\mathrm T}-(Re_3)^{\mathrm T})P_{\vartheta+e_3}=-(Re_3)^{\mathrm T}P_{\vartheta+e_3}, \]
which follows from the definition \eqref{eqRotatedDetector} of $\vartheta_R$, applying $P^{\mathrm T}$ from the left and $P$ from the right, and using as before $R^{\mathrm T}P^{\mathrm T}P=P_{R^{\mathrm T}e_3}R^{\mathrm T}$, the equation \eqref{eqPRequation} yields
\begin{multline}\label{eqProjectionIdentityForXOrthogonal}
-\vartheta_{R,3}P_{R^{\mathrm T}e_3}\left(\frac{\vartheta+e_3}{|\vartheta+e_3|^2}(\vartheta+e_3)^{\mathrm T}X\right)P_{R^{\mathrm T}e_3} \\
= R^{\mathrm T}P^{\mathrm T}B_RPR-P_{R^{\mathrm T}e_3}(\mathbbm 1-\alpha_R(\vartheta+e_3)\vartheta_R^{\mathrm T})P_{\vartheta+e_3}XP_{R^{\mathrm T}e_3}.
\end{multline}

Since the right hand side is already known (it depends only on $P_{\vartheta+e_3}X$), the equation system~\eqref{eqProjectionIdentityForXOrthogonal} for $R=R_0,R_1,R_2$ can be uniquely solved for
\[ \frac{\vartheta+e_3}{|\vartheta+e_3|^2}(\vartheta+e_3)^{\mathrm T}X. \]

Therefore, the equations \eqref{eqReconstructionAnisotropic} uniquely determine $X$ and because $A$ is by construction a solution of the equations, this implies that $X=A$.
\end{proof}

Thus, applying \autoref{thReconstructionAnisotropic} to the matrix $A=\bar\chi(\tau;\sigma,\vartheta)$ shows that the measurements $a_{R,p,j}$ obtained at the detectors $\vartheta_R$ for the polarisations $p=e_1,e_2,e_1+e_2$ and rotations $R=R_0,R_1,R_2$, fulfilling the assumptions of \autoref{thReconstructionAnisotropic}, provide sufficient information to reconstruct the Radon data $\bar\chi(\tau;\sigma,\vartheta)$. Calculating these two dimensional Radon data for all directions $\vartheta$ in some subset of $S^2_+$ (by considering some additional rotations so that for every direction $\vartheta$, there exist three rotations fulfilling the assumptions of \autoref{thReconstructionAnisotropic}), it is possible via an inversion of a limited angle Radon transform to finally recover the susceptibility $\chi$.

\section{Conclusions}

In this chapter, a general mathematical model of OCT based on Maxwell's equations has been presented.  As a consequence of this modelling, OCT was formulated as an inverse scattering problem for the susceptibility $\chi$. It was shown that without additional assumptions about the optical properties of the medium, in general, $\chi$ cannot be reconstructed due to lack of measurements. 
Some reasonable physical assumptions were presented, under which the medium can, in fact, be reconstructed. For instance, if the medium is isotropic, iterative schemes to reconstruct the susceptibility were developed. Dispersion and focus illumination are also considered. For an anisotropic medium, it follows that different incident fields, with respect to direction (rotating the sample) and polarisation,  should be considered to completely recover  $\chi.$ 

\section*{Acknowledgements}
The authors would like to thank Wolfgang Drexler and Boris Hermann 
from the Medical University Vienna for their valuable comments and stimulating discussions. 
This work has been supported by the Austrian Science Fund (FWF) within the national research network Photoacoustic Imaging in Biology and Medicine, projects S10501-N20 and S10505-N20.

\section{Cross - References}
\begin{itemize}
\item Inverse Scattering
\item Optical Imaging
\item Tomography
\item Wave Phenomena
\end{itemize}

\bibliographystyle{plain}

\begin{thebibliography}{10}

\bibitem{AmmBao01}
H.~Ammari and G.~Bao.
\newblock Analysis of the scattering map of a linearized inverse medium problem
  for electromagnetic waves.
\newblock {\em Inverse Probl.}, 17:219--234, 2001.

\bibitem{AndThrYurTycJorFro04}
P.~E. Andersen, L~Thrane, H.~T. Yura, A.~Tycho, T.~M. J{\o}rgensen, and M.~H.
  Frosz.
\newblock Advanced modelling of optical coherence tomography systems.
\newblock {\em Phys. Med. Biol.}, 49:1307--1327, 2004.

\bibitem{BorWol99}
M.~Born and E.~Wolf.
\newblock {\em Principles of Optics (7th Edition)}.
\newblock Cambridge University Press, Cambridge, 1999.

\bibitem{BouTea02}
B.~E. Bouma and G.~J. Tearney.
\newblock {\em Handbook of Optical Coherence Tomography}.
\newblock Marcel Dekker, Inc., New York, Basel, 2002.

\bibitem{Bre06}
M.~E. Brezinski.
\newblock {\em Optical Coherence Tomography Principles and Applications}.
\newblock Academic Press, New York, 2006.

\bibitem{BroThuBur00}
A.~Brodsky, S.~R. Thurber, and L.~W. Burgess.
\newblock Low-coherence interferometry in random media. i. theory.
\newblock {\em J. Opt. Soc. Amer. A}, 17(11):2024--2033, 2000.

\bibitem{BruCha05}
O.~Bruno and J.~Chaubell.
\newblock One-dimensional inverse scattering problem for optical coherence
  tomography.
\newblock {\em Inverse Probl.}, 21:499--524, 2005.

\bibitem{ColKre98}
D.~Colton and R.~Kress.
\newblock {\em Inverse acoustic and electromagnetic scattering theory},
  volume~93 of {\em Applied Mathematical Sciences}.
\newblock Springer-Verlag, Berlin, second edition, 1998.

\bibitem{Dol98}
L.~S. Dolin.
\newblock A theory of optical coherence tomography.
\newblock {\em Radiophys. Quantum Electron.}, 41(10):850--873, 1998.

\bibitem{DreFuj08}
W.~Drexler and J.~G. Fujimoto.
\newblock {\em Optical Coherence Tomography}.
\newblock Springer, Berlin, Heidelberg, 2008.

\bibitem{DuaMakYamLimYas11}
L.~Duan, S.~Makita, M.~Yamanari, Y.~Lim, and Y.~Yasuno.
\newblock Monte-carlo-based phase retardation estimator for polarization
  sensitive optical coherence tomography.
\newblock {\em Opt. Express}, 19:16330--16345, 2011.

\bibitem{FenWanEld03}
Y.~Feng, R.~K. Wang, and J.~B. Elder.
\newblock Theoretical model of optical coherence tomography for system
  optimization and characterization.
\newblock {\em J. Opt. Soc. Amer. A}, 20(9):1792--1803, 2003.

\bibitem{Fer96}
A.~F. Fercher.
\newblock Optical coherence tomography.
\newblock {\em J. Biomed. Opt.}, 1(2):157--173, 1996.

\bibitem{Fer10}
A.~F. Fercher.
\newblock Optical coherence tomography - development, principles, applications.
\newblock {\em Zeitschrift f{\"u}r Medizinische Physik}, 20:251--276, 2010.

\bibitem{FerDreHitLas03}
A.~F. Fercher, W.~Drexler, C.~K. Hitzenberger, and T.~Lasser.
\newblock Optical coherence tomography - principles and applications.
\newblock {\em Rep. Prog. Phys.}, 66(2):239--303, 2003.

\bibitem{FerHit02}
A.~F. Fercher and C.~K. Hitzenberger.
\newblock {\em Optical Coherence Tomography}.
\newblock Progress in Optics. Elsevier Science B. V., Amsterdam, 2002.

\bibitem{FerHitDreKamSat93}
A.~F. Fercher, C.~K. Hitzenberger, W.~Drexler, G.~Kamp, and H.~Sattmann.
\newblock In vivo optical coherence tomography.
\newblock {\em Amer. J. Ophthal.}, 116:113--114, 1993.

\bibitem{FerHitKamZai95}
A.~F. Fercher, C.~K. Hitzenberger, G.~Kamp, and S.~Y. El~Zaiat.
\newblock Measurement of intraocular distances by backscattering spectral
  interferometry.
\newblock {\em Opt. Commun.}, 117:43--48, 1995.

\bibitem{FerSanJorAnd09}
A.~F. Fercher, B.~Sander, T.~M. J{\o}rgensen, and P.~E. Andersen.
\newblock {\em Optical Coherence Tomography}.
\newblock Encyclopedia of analytical chemistry. John Wiley \& Sons Ltd., 2009.

\bibitem{FriWol83}
A.~T. Friberg and E.~Wolf.
\newblock Angular spectrum representation of scattered electromagnetic fields.
\newblock {\em J. Opt. Soc. Amer.}, 73(1):26--32, 1983.

\bibitem{Hec02}
E.~Hecht.
\newblock {\em Optics}.
\newblock Addison Wesley, fourth edition, 2002.

\bibitem{Hel96}
T.~Hellmuth.
\newblock Contrast and resolution in optical coherence tomography.
\newblock In I.~J. Bigio, W.~S. Grundfest, H.~Schneckenburger, K.~Svanberg, and
  P.~M. Viallet, editors, {\em Optical Biopsies and Microscopic Techniques},
  volume 2926 of {\em Proc. SPIE}, pages 228--237, 1997.

\bibitem{Hoh06}
T.~Hohage.
\newblock Fast numerical solution of the electromagnetic medium scattering
  problem and applications to the inverse problem.
\newblock {\em J. Comput. Phys.}, 214:224--238, 2006.

\bibitem{HuaSwaLinSchuStiCha91}
D.~Huang, E.~A. Swanson, C.~P. Lin, J.~S. Schuman, G.~Stinson, W.~Chang, M.~R.
  Hee, T.~Flotte, K.~Gregory, C.~A. Puliafito, and J.~G. Fujimoto.
\newblock Optical coherence tomography.
\newblock {\em Science}, 254(5035):1178--1181, 1991.

\bibitem{IzaCho08}
J.~A. Izatt and M.~A. Choma.
\newblock Theory of optical coherence tomography.
\newblock In W.~Drexler and J.~G. Fujimoto, editors, {\em Optical Coherence
  Tomography}, pages 47--72. Springer, 2008.

\bibitem{KirMegKuzSerMyl10}
M.~Kirillin, I.~Meglinski, V.~Kuzmin, E.~Sergeeva, and R.~Myllyl{\"a}.
\newblock Simulation of optical coherence tomography images by monte carlo
  modeling based on polarization vector approach.
\newblock {\em Opt. Express}, 18(21):21714--21724, 2010.

\bibitem{KnuSchoBoc96}
A.~Kn{\"u}ttel, R.~Schork, and D.~B{\"o}cker.
\newblock Analytical modeling of spatial resolution curves in turbid media
  acquired with optical coherence tomography (oct).
\newblock In C.~J. Cogwell, G.~S. Kino, and T.~Wilson, editors, {\em Three-
  Dimensional Microscopy: Image Acquisition and Processing III}, volume 2655 of
  {\em Proc. SPIE}, pages 258--270, 1996.
  
  \bibitem{MarDavBopCar09}
D.~L. Marks, B.~J. Davis, S.~A. Boppart, and P.~S. Carney.
\newblock Partially coherent illumination in full-field interferometric
  synthetic aperture microscopy.
\newblock {\em J. Opt. Soc. Amer. A}, 26(2):376--386, 2009.

\bibitem{MarRalBopCar07}
D.~L. Marks, T.~S. Ralston, S.~A. Boppart, and P.~S. Carney.
\newblock Inverse scattering for frequency-scanned full-field optical coherence
  tomography.
\newblock {\em J. Opt. Soc. Amer. A}, 24(4):1034--1041, 2007.

\bibitem{Orf02}
S.~J. Orfanidis.
\newblock {\em Electromagnetic Waves and Antennas}.
\newblock Rutgers University Press, 2002.

\bibitem{PanBirRosEng95}
Y.~Pan, R.~Birngruber, J.~Rosperich, and R.~Engelhardt.
\newblock Low-coherence optical tomography in turbid tissue: theoretical
  analysis.
\newblock {\em App. Opt.}, 34(28):6564--6574, 1995.

\bibitem{Pod05}
A.~G. Podoleanu.
\newblock Optical coherence tomography.
\newblock {\em Brit. J. Radiology}, 78:976--988, 2005.

\bibitem{Pot00}
R.~Potthast.
\newblock Integral equation methods in electromagnetic scattering from
  anisotropic media.
\newblock {\em Math. Methods Appl. Sci.}, 23:1145--1159, 2000.

\bibitem{RalMarKamBop05}
T.~S. Ralston.
\newblock Deconvolution methods for mitigation of transverse blurring in
  optical coherence tomography.
\newblock {\em IEEE Trans. Image Process.}, 14(9):1254--1264, 2005.

\bibitem{RalMarCarBop06}
T.~S. Ralston, D.~L. Marks, P.~S. Carney, and S.~A. Boppart.
\newblock Inverse scattering for optical coherence tomography.
\newblock {\em J. Opt. Soc. Amer. A}, 23(5):1027--1037, 2006.

\bibitem{Schm99}
J.~M. Schmitt.
\newblock Optical coherence tomography {(OCT)}: {A} review.
\newblock {\em IEEE J. Quantum Electron.}, 5:1205--1215, 1999.

\bibitem{SchmKnu97}
J.~M. Schmitt and A.~Kn{\"u}ttel.
\newblock Model of optical coherence tomography of heterogeneous tissue.
\newblock {\em J. Opt. Soc. Amer. A}, 14(6):1231--1242, 1997.

\bibitem{SchmKnuBon93}
J.~M. Schmitt, A.~Kn{\"u}ttel, and R.~F. Bonner.
\newblock Measurement of optical properties of biological tissues by
  low-coherence reflectometry.
\newblock {\em App. Opt.}, 32:6032--6042, 1993.

\bibitem{SchmXiaYun98}
J.~M. Schmitt, S.~H. Xiang, and K.~M. Yung.
\newblock Differential absorption imaging with optical coherence tomography.
\newblock {\em J. Opt. Soc. Amer. A}, 15:2288--2296, 1998.

\bibitem{SmiLinCheNelMil98}
D.~J. Smithies, T.~Lindmo, Z.~Chen, J.~S. Nelson, and T.~E. Milner.
\newblock Signal attenuation and localization in optical coherence tomography
  studied by monte carlo simulation.
\newblock {\em Phys. Med. Biol.}, 43:3025--3044, 1998.

\bibitem{SwaIzaHeeHuaLinSchu93}
E.~A. Swanson, J.~A. Izatt, M.~R. Hee, D.~Huang, C.~P. Lin, J.~S. Schuman,
  C.~A. Puliafito, and J.~G. Fujimoto.
\newblock In vivo retinal imaging by optical coherence tomography.
\newblock {\em Opt. Letters}, 18:1864--1866, 1993.

\bibitem{ThoSanMogThrJorJem09}
J.~B. Thomsen, B.~Sander, M.~Mogensen, L.~Thrane, T.~M. J{\o}rgensen,
  T.~Martini, G.~B.~E. Jemec, and P.~E. Andersen.
\newblock Optical coherence tomography: Technique and applications.
\newblock In {\em Advanced Imaging in Biology and Medicine}, pages 103--129.
  Berlin, Heidelberg, Springer, 2009.

\bibitem{ThrYurAnd00}
L.~Thrane, H.~T. Yura, and P.~E. Andersen.
\newblock Analysis of optical coherence tomography systems based on the
  extended huygens - fresnel principle.
\newblock {\em J. Opt. Soc. Amer. A}, 17(3):484--490, 2000.

\bibitem{TomWan05}
P.~H. Tomlins and R.~K. Wang.
\newblock Theory, developments and applications of optical coherence
  tomography.
\newblock {\em J. Phys. D: Appl. Phys.}, 38:2519--2535, 2005.

\bibitem{TurSerDolKamSha05}
I.~V. Turchin, E.~A. Sergeeva, L.~S. Dolin, V.~A. Kamensky, N.~M. Shakhova, and
  R.~Richards~Kortum.
\newblock Novel algorithm of processing optical coherence tomography images for
  differentiation of biological tissue pathologies.
\newblock {\em J. Biomed. Opt.}, 10(6):064024, 2005.

\bibitem{XuMarDoBop04}
C.~Xu, D.~L. Marks, M.~N. Do, and S.~A. Boppart.
\newblock Separation of absorption and scattering profiles in spectroscopic
  optical coherence tomography using a least-squares algorithm.
\newblock {\em Opt. Express}, 12(20):4790--4803, 2004.

\end{thebibliography}

\end{document}